\pdfoutput=1
\documentclass[reqno,11pt,final]{amsart}
\usepackage{glaudo_pkg_en}

\title{Minkowski inequality for nearly spherical domains}

\author{F. Glaudo}
\address{ETH, Rämistrasse 101, 8092 Zürich, Switzerland}
\email{federico.glaudo@math.ethz.ch}

\let\lamdba\lambda
\DeclareMathOperator{\II}{I\hspace{-0.5pt}I}

\begin{document}
\begin{abstract}
    We investigate the validity and the stability of various Minkowski-like inequalities for $C^1$-perturbations of the ball.

    Let $K\subseteq\R^n$ be a domain (possibly not convex and not mean-convex) which is $C^1$-close to a ball. 
    We prove the sharp geometric inequality
    \begin{equation*}
        \left(\int_{\partial K} \norm{\II}_1\de\Haus^{n-1}\right)^{\frac1{n-2}} \ge C_1(n)\Per(K)^{\frac1{n-1}} \,,
    \end{equation*}
    where $C_1(n)$ is the constant that yields the equality when $K=B_1$ (and $\norm{\II}_1$ is the sum of the absolute values of the eigenvalues of the second fundamental form $\II$ of $\partial K$). 
    Moreover, for any $\delta>0$, if $K$ is sufficiently $C^1$-close to a ball, we show the almost sharp Minkowski inequality
    \begin{equation*}
        \left(\int_{\partial K} H^+\de\Haus^{n-1}\right)^{\frac1{n-2}} \ge (C_1(n)-\delta)\Per(K)^{\frac1{n-1}} \,.
    \end{equation*}
    If $K$ is axially symmetric, we prove the Minkowski inequality with the sharp constant (i.e., $\delta=0$).
    
    We establish also the sharp quantitative stability (in the family of $C^1$-perturbations of the ball) of the volumetric Minkowski inequality
    \begin{equation}\label{eq:isop_abstract}
        \left(\int_{\partial K} H^+\de\Haus^{n-1}\right)^{\frac1{n-2}} \ge C_2(n)\abs{K}^{\frac1n} \,,
    \end{equation}
    where $C_2(n)$ is the constant that yields the equality when $K=B_1$. More precisely, we control the deviation of $K$ from a ball (in a strong norm) with the difference between the left-hand side and the right-hand side of \cref{eq:isop_abstract}.
    
    Finally, we show, by constructing a counterexample, that the mentioned inequalities are false (even for domains $C^1$-close to the ball) if one replaces $H^+$ with $H$.
\end{abstract}

\maketitle

\section{Introduction}
\subsection{Background}\label{subsec:background}
Given an open bounded domain $K\subseteq\R^n$ with smooth boundary, we denote with $\partial K$ its boundary, with $\II$ the second fundamental form of the boundary and with $H=\tr(\II)$ the mean curvature of the boundary. Moreover, for any $k=1,\dots, n-1$, let $\sigma_k(\II)$ be the $k$-th symmetric function of the eigenvalues of $\II$ (e.g., $\sigma_1(\II)=H$).

As a consequence of the Alexandrov--Fenchel inequalities (see \cite[Theorem 7.3.1]{Schneider2014}), we have the following inequalities for the \emph{quermassintegrals}\footnote{With the name quermassintegral we refer to the integrals of the symmetric functions of the curvatures of a hypersurface (in this case, the boundary of $K$). When $K$ is convex, the quermassintegrals coincide (up to a normalization) with some suitable mixed volumes involving only $K$ and the unit ball (see \cite[(5.53), (5.55)]{Schneider2014}).
We write $\abs{\sigma_k(\II)}$ instead of $\sigma_k(\II)$ as we will later discuss this inequality for non convex domains.} of a \emph{convex} body $K$ (see \cite[(7.67)]{Schneider2014})
\begin{equation}\label{eq:general_quermassintegral_ineq}
    \frac{\left(\int_{\partial K} \abs{\sigma_k(\II)}\de\Haus^{n-1}\right)^{\frac1{n-1-k}}}
    {\left(\int_{\partial K} \abs{\sigma_{k-1}(\II)}\de\Haus^{n-1}\right)^{\frac1{n-k}}}
    \ge
    \frac{\left(\int_{\S^{n-1}} \sigma_k(\id_{n-1})\de\Haus^{n-1}\right)^{\frac1{n-1-k}}}
    {\left(\int_{\S^{n-1}} \sigma_{k-1}(\id_{n-1})\de\Haus^{n-1}\right)^{\frac1{n-k}}}
    \quad \text{for all $k=1,\dots, n-1$. }
\end{equation}
In other words, if $\int_{\partial K}\abs{\sigma_{k-1}(\II)}$ is fixed, the minimum of $\int_{\partial K}\abs{\sigma_k(\II)}$ in the family of convex sets is achieved when $K$ is a ball. 
The case $k=0$ of \cref{eq:general_quermassintegral_ineq} (which is not well-defined with the above notation) corresponds to the classical isoperimetric inequality
\begin{equation}\label{eq:isop_ineq}
\frac{\Per(K)^{\frac1{n-1}}}{\abs{K}^{\frac1n}}
\ge
\frac{\Per(B_1)^{\frac1{n-1}}}{\abs{B_1}^{\frac1n}} \,.
\end{equation}
Concatenating the isoperimetric inequality with \cref{eq:general_quermassintegral_ineq}, one obtains
\begin{equation}\label{eq:quermassintegral_volume}
    \frac{\left(\int_{\partial K} \abs{\sigma_k(\II)}\de\Haus^{n-1}\right)^{\frac1{n-1-k}}}
    {\abs{K}^{\frac1n}}
    \ge
    \frac{\left(\int_{\S^{n-1}} \sigma_k(\id_{n-1})\de\Haus^{n-1}\right)^{\frac1{n-1-k}}}
    {\abs{B_1}^{\frac1n}}
    \quad \text{for all $k=1,\dots, n-1$. }
\end{equation}

While the isoperimetric inequality \cref{eq:isop_ineq} is well-known for arbitrary domains in $\R^n$ (as its stability, see \cite{FuscoMaggiPratelli2008}), this is not the case for the other inequalities between quermassintegrals. 
Let us give a schematic review of the literature on the topic.
\begin{itemize}
    \item For a $k_0$-convex (i.e., $\sigma_i(\II)\ge 0$ for all $1\le k\le k_0$) star-shaped domain $K$, the inequalities \cref{eq:general_quermassintegral_ineq} with $k\le k_0$ are proven in \cite{GuanLi2009} via a suitable flow of the hypersurface $\partial K$ that converges to a sphere (see also \cite{Gerhardt1990,Urbas1990}).
    \item For a $(k_0+1)$-convex domain, the inequalities \cref{eq:general_quermassintegral_ineq} with $k\le k_0$ are proven in \cite{ChangWang2013} with the method of optimal transport albeit with a non-sharp constant (that is, the right-hand side is replaced by a positive constant).
    \item In dimension $n=3$, the inequality \cref{eq:general_quermassintegral_ineq} holds for axially-symmetric domains \cite{DalphinHenrotMasnouTakahashi2016}.
    \item For $(k_0+1)$-convex domains in $\R^n$, the inequalities \cref{eq:quermassintegral_volume} with $k\le k_0$ have been established in \cite{ChangWang2014,Qiu2015} with optimal transport techniques.
    \item For an outward-minimizing (which is a more restrictive property than mean-convexity) domain $K$, the case $k=1$ of \cref{eq:general_quermassintegral_ineq} is due to Huisken (and reported in \cite[Theorem 5, Lemma 8]{FreierSchwartz2014}, see also \cite[Theorem 1.1]{Wei2018}). The proof uses in a fundamental way the theory of the (weak) inverse mean curvature flow developed in \cite{HuiskenIlmanen2001}. The same result is also shown in \cite{AgostinianiFogagnoloMazzieri2021} with an alternative proof which avoids all the regularity issues which arise when applying the inverse mean curvature flow in dimension higher than $7$.
    As a direct consequence of this result, one deduces \cref{eq:quermassintegral_volume} with $k=1$ for \emph{arbitrary} domains $K$ in $\R^n$.
    \item The case $k=1$ of \cref{eq:general_quermassintegral_ineq} for general domains was established, with a non-sharp constant, in \cite{MichaelSimon1973}. Subsequently \cite{Castillon2010} proved again the result with an improved constant (via optimal transport methods) and recently Brendle found a proof via the ABP method that provides the sharp constant for hypersurfaces with boundary\footnote{Brendle proves the result for a generic $\Sigma\subseteq\R^n$ (possibly with boundary) instead of the more restricted family $\Sigma=\partial K$. In the family he considers, the optimal shapes for the inequality are flat disks (instead of spheres) and thus the constant he obtains is not the sharp one in our setting.} (see \cite{Brendle2019}).
\end{itemize}

As it shall be clear from this quick review of the literature, the validity of \cref{eq:general_quermassintegral_ineq} for general domains is completely open (if the sharp constant is desired).

We will focus on \cref{eq:general_quermassintegral_ineq} and \cref{eq:quermassintegral_volume} with $k=1$, known in the literature respectively as Minkowski inequality 
\begin{equation}\label{eq:minkowski}
    \frac{\left(\int_{\partial K} H^+\de\Haus^{n-1}\right)^{\frac1{n-2}}}
    {\Per(K)^{\frac1{n-1}}}
    \ge
    \frac{\left((n-1)\Haus^{n-1}(\S^{n-1})\right)^{\frac1{n-2}}}
    {\Per(B_1)^{\frac1{n-1}}}
\end{equation}
and \emph{volumetric} Minkowski inequality
\begin{equation}\label{eq:minkowski_vol}
    \frac{\left(\int_{\partial K} H^+\de\Haus^{n-1}\right)^{\frac1{n-2}}}
    {\abs{K}^{\frac1{n}}}
    \ge
    \frac{\left((n-1)\Haus^{n-1}(\S^{n-1})\right)^{\frac1{n-2}}}
    {\abs{B_1}^{\frac1{n}}} \,,
\end{equation}
where $H^+\defeq \max(H, 0)$ denotes the positive part of the mean curvature\footnote{Here we are silently replacing $\abs{H}$, which was appearing in \cref{eq:general_quermassintegral_ineq,eq:quermassintegral_volume}, with $H^+$. The inequalities with $H^+$ are stronger and they are the ones we are going to focus on. Let us mention that, with the inverse mean curvature flow, one proves \cref{eq:minkowski_vol} and then \cref{eq:quermassintegral_volume} with $k=1$ follows as a consequence.}.

The topic of the present paper is to investigate the validity and the stability of these Minkowski-type inequalities for domains $K$ that are $C^1$-perturbations of the ball.
Our investigation was inspired by the work \cite{Fuglede89} where a similar analysis is performed for the isoperimetric inequality (establishing its stability in the family of $C^1$-perturbations of the ball). Several new ideas are necessary since the functionals we consider (e.g., the integral of the mean curvature) depend on the curvatures of $\partial K$ but we do not assume any control on such curvatures (hence a Taylor expansion \`a la Fuglede up to second order is not feasible).

In \cite[Appendix E]{ChodoshEichmair2022}, a Minkowski inequality with an additional term which depends on the trace-free part of the second fundamental form is established in dimension $n=3$ for $C^1$-perturbations (see \cite[Proposition E.4]{ChodoshEichmair2022}) with methods similar to ours.

In order to put in the right context our work, let us briefly review the work of Fuglede.

\subsubsection{Fuglede's work on the stability of the isoperimetric inequality}
Given a star-shaped domain $K\subseteq\R^n$, let $u:\S^{n-1}\to(-1,\infty)$ be the function such that $(1+u(x))x \in\partial K$ for any $x\in\S^{n-1}$. We assume that $\norm{u}_{C^1}\ll 1$. In this setting, one has the following formulas for the volume and the perimeter of $K$ (see \cref{lem:formulas}):
\begin{align*}
    \abs{K} &= \frac1n\int_{\S^{n-1}} (1+u)^n 
    =  \abs{B_1}+\int_{\S^{n-1}}u + \frac{n-1}2\int_{\S^{n-1}}u^2 +\smallo\Big(\int u^2+\abs{\nabla u}^2\Big)
    \,,\\
    \Per(K) &= 
    \int_{\S^{n-1}} (1+u)^{n-2}\sqrt{(1+u)^2+\abs{\nabla u}^2} \\
     &=
    \Per(B_1) + \int_{\S^{n-1}}(n-1)u + \frac{(n-1)(n-2)}{2}u^2
    +\frac12\abs{\nabla u}^2
    +\smallo\Big(\int u^2+\abs{\nabla u}^2\Big)
    \,.
\end{align*}
Thus, if we assume $\abs{K}=\abs{B_1}$ (which is true up to scaling), we find that the isoperimetric deficit can be written as
\begin{equation}\label{eq:isop_deficit_fuglede}
    \Per(K)-\Per(B_1) = \frac12 \int_{\S^{n-1}}\abs{\nabla u}^2 - (n-1)u^2 +\smallo\Big(\int u^2+\abs{\nabla u}^2\Big) \,.
\end{equation}
Since $(n-1)$ is the first (nonzero) eigenvalue of the Laplacian on the sphere, the left-hand side of \cref{eq:isop_deficit_fuglede} may be negative in general. 
The only missing observation is that we may assume (up to translation) that the barycenter of $K$ is the origin. From this assumption we deduce that $u$ is (almost) orthogonal to all the eigenfunctions of the Laplacian with eigenvalue equal to $n-1$ and therefore $u$ satisfies the stronger Poincaré inequality (recall that $2n$ is the next eigenvalue after $n-1$)
\begin{equation*}
    \int_{\S^{n-1}}\abs{\nabla u}^2 \ge (2n-\smallo(1))\int_{\S^{n-1}}u^2 \,.
\end{equation*}
This latter inequality, together with \cref{eq:isop_deficit_fuglede} is sufficient to prove the stability of the isoperimetric inequality for domains $C^1$-close to $B_1$. In other words the isoperimetric deficit
\begin{equation*}
    \frac{\Per(K)^{\frac1{n-1}}}{\abs{K}^{\frac1n}}
    -
    \frac{\Per(B_1)^{\frac1{n-1}}}{\abs{B_1}^{\frac1n}}
\end{equation*}
controls (quantitatively and in a strong norm) the distance between $K$ and a ball (with center equal to the barycenter of $K$ and volume equal to the volume of $K$).

\subsection{Results}
Our plan is to repeat Fuglede's argument for the Minkowski-type inequalities \cref{eq:minkowski,eq:minkowski_vol}.
Let us remark that, differently from what happens in Fuglede's argument, since we consider quantities that depend on the curvatures of $\partial K$, the assumption of $C^1$-closeness to the ball does not control our quantities. 

We say that a domain $K\subseteq\R^n$ (with smooth boundary) is a $C^1$ $\eps$-perturbation of a ball if, for some $\bar x\in\R^n$, 
\begin{equation}\label{eq:def_eps_perturbation}
    \norm*{\nu(x)-\frac{x-\bar x}{\abs{x-\bar x}}}_{L^{\infty}(\partial K)} \le \eps \,,
\end{equation}
where $\nu(x)$ denotes the normal vector to $\partial K$ at the point $x\in\partial K$.

Let us remark that the boundary $\partial K$ of a $C^1$ $\eps$-perturbation of a ball may have unbounded (both positive and negative) curvatures and may not be convex or mean-convex. 
Notice that if $K$ is a $C^1$ $\eps$-perturbation of a ball with $\bar x$ and $\eps$ sufficiently small, then it is star-shaped.

We prove the following sharp Minkowski-like inequality for perturbations of a ball.
\begin{theorem}\label{thm:main_statement_nuclear}
    Given $n\ge 3$, there is an $\eps_1=\eps_1(n)$ such that the following statement holds.
    If $K\subseteq\R^n$ is a $C^1$ $\eps_1$-perturbation of a ball, it holds
    \begin{equation}\label{eq:main_statement_nuclear}
        \frac{\left(\int_{\partial K} \norm{\II}_1 \de\Haus^{n-1}\right)^{\frac1{n-2}}}
        {\Per(K)^{\frac{1}{n-1}}}
        \ge
        \frac{\left((n-1)\Haus^{n-1}(\S^{n-1})\right)^{\frac1{n-2}}}
        {\Per(B_1)^{\frac{1}{n-1}}} \,,
    \end{equation}
    where $\norm{\II}_1$ is the \emph{nuclear norm} of the second fundamental form of $\partial K$ (i.e., the sum of the absolute values of its eigenvalues).
\end{theorem}
The estimate \cref{eq:main_statement_nuclear} is weaker than \cref{eq:minkowski} because $\norm{\II}_1 \ge |H|$ (with equality if and only if $K$ is convex).

However, under the assumption of axial symmetry, we can prove the sharp Minkowski inequality.
\begin{theorem}\label{thm:main_statement_axial}
    Given $n\ge 3$, there is an $\eps_2=\eps_2(n)$ such that the following statement holds. If $K\subseteq\R^n$ is a $C^1$ $\eps_2$-perturbation of a ball and it is axially symmetric (i.e., it is invariant under rotations around a fixed line), then
    \begin{equation*}
        \frac{\Big(\int_{\partial K} H^+\de \Haus^{n-1}\Big)^{\frac1{n-2}}}{\Per(K)^{\frac1{n-1}}}
        \ge
        \frac{\left((n-1)\Haus^{n-1}(\S^{n-1})\right)^{\frac1{n-2}}}
        {\Per(B_1)^{\frac{1}{n-1}}} \,,
    \end{equation*}
    where $H^+=\max(H, 0)$.
\end{theorem}
Let us remark that, even if the statement of \cref{thm:main_statement_axial} is \emph{morally} one-dimensional, our proof uses crucially various tools (e.g., the Poincaré inequality on the sphere, the decomposition in eigenfunctions of $-\lapl$) which are not one-dimensional.

We show also the sharp quantitative stability of \cref{eq:minkowski_vol} for  perturbations of a ball.
\begin{theorem}\label{thm:main_statement_stability}
    Given $n\ge 4$, there are two constants $\eps_3=\eps_3(n)>0$ and $c(n)>0$ such that the following statement holds. 
    If $K\subseteq\R^n$ is a $C^1$ $\eps_3$-perturbation of a ball, it holds, for some $\bar x\in\R^n$,
    \begin{equation}\label{eq:main_statement_stability}
        \frac{\left(\int_{\partial K} H^+ \de\Haus^{n-1}\right)^{\frac1{n-2}}}
        {\abs{K}^{\frac{1}{n}}}
        -
        \frac{\left((n-1)\Haus^{n-1}(\S^{n-1})\right)^{\frac1{n-2}}}
        {\abs{B_1}^{\frac{1}{n}}}
        \ge c(n) \fint_{\partial K} \abs*{\nu - \frac {x-\bar x}{\abs{x-\bar x}}}^2 \de\Haus^{n-1} \,,
    \end{equation}
    where $H^+=\max(0, H)$.
\end{theorem}

Finally, we prove the Minkowski inequality for $C^1$-perturbations of the ball but with an \emph{almost sharp} constant.

\begin{theorem}\label{thm:main_statement_almost_sharp}
    Given $n\ge 3$, for any $\delta>0$ there is an $\eps_4=\eps_4(n, \delta)$ such that the following statement holds.
    If $K\subseteq\R^n$ is a $C^1$ $\eps_4$-perturbation of a ball, it holds
    \begin{equation*}
        \frac{\left(\int_{\partial K} H^+ \de\Haus^{n-1}\right)^{\frac1{n-2}}}
        {\Per(K)^{\frac{1}{n-1}}}
        \ge
        \frac{\left((n-1)\Haus^{n-1}(\S^{n-1})\right)^{\frac1{n-2}}}
        {\Per(B_1)^{\frac{1}{n-1}}} - \delta \, .
    \end{equation*}
\end{theorem}

In \cref{app:counterexample} we construct a domain $K$ which is $C^1$-close to the ball $B_1$ but such that $\int_{\partial K}H<0$. This shows that, already in the perturbative setting, this Minkowski-like inequalities are rather delicate. In particular all our results would be false if instead of $H^+$ we were to use $H$. On the other hand, let us remark that all our results are already known for $C^2$-perturbations of a ball since $C^2$-perturbations of a ball are convex.
On the other hand for $C^1$-perturbations the results are new and much more delicate.

Even though our results are the best available without assuming mean-convexity, we are not able to prove the ideal statement one might hope to show with our methods. Indeed, the ideal statement is the sharp Minkowski inequality for $C^1$-perturbations of a ball. 
There is a major obstruction which prevents us from proving the optimal statement.
We discuss this obstruction and state a related conjecture (which might be of independent interest) in \cref{sec:open_conjecture}.

As will be clear from the proofs, even if we are not stating it, we are able to show also the quantitative sharp stability of \cref{thm:main_statement_nuclear} and \cref{thm:main_statement_axial}.

\subsection{Outline of the proofs}
We describe the main steps of the proof of \cref{thm:main_statement_stability} (which is the simplest one) under the additional constraint $H\ge 0$. The proofs of \cref{thm:main_statement_nuclear,thm:main_statement_axial,thm:main_statement_almost_sharp} go along similar lines.
Let us emphasize that the real proofs require additional ideas (necessary to remove the assumption $H\ge 0$) and are much more technically involved compared to what we are going to describe.
For the sake of clarity, in this outline we will sweep under the carpet many difficulties and we will heavily simplify some steps by considering simpler expressions in place of the ones which shall be considered.

We are going to adopt the notation introduced in \cref{sec:notation}.

Let $K\subseteq\R^n$ be a $C^1$ $\eps$-perturbation of the ball (with $\eps>0$ sufficiently small) which is mean-convex (i.e., $H\ge 0$). Then, up to translation, there is a function $u:\S^{n-1}\to[-1,\infty]$ such that $(1+u(x))x\in\partial K$ for all $x\in\S^{n-1}$ and $\norm{u}_{C^1}\lesssim \eps$ (see \cref{lem:nablau_vs_normal}).

It holds (see \cref{eq:integral_H_formula})
\begin{equation*}\begin{aligned}
    &\int_{\partial K} H - (n-1)\Haus^{n-1}(\S^{n-1})
    =
    (n-1)(n-2)\int_{\S^{n-1}} u + \frac{(n-1)(n-2)(n-3)}{2}\int_{\S^{n-1}} u^2 \\
    &\hspace{155pt}+ (n-2)\int_{\S^{n-1}}\abs{\nabla u}^2
    + \int_{\S^{n-1}} (1+u)^{n-3}\frac{\nabla^2 u[\nabla u,\nabla u]}{(1+u)^2 + \abs{\nabla u}^2}
    \\
    &\hspace{155pt}+\omega(\eps)\int_{\S^{n-1}} u^2 + \abs{\nabla u}^2 \,,
\end{aligned}\end{equation*}
where $\omega(\eps)$ denotes a quantity that goes to $0$ as $\eps\to0$.
If we assume that the barycenter of $K$ is $0_{\R^n}$ and $|K|=|B_1|$, then the methods developed by Fuglede (as described in \cref{subsec:background}) coupled with the previous identity imply
\begin{equation}\label{eq:intro_with_cubicterm}\begin{aligned}
    \int_{\partial K} H - (n-1)\Haus^{n-1}(\S^{n-1})
    &\ge
     (n-2-\omega(\eps))\int_{\S^{n-1}}\abs{\nabla u}^2  - (n-1)(n-2)\int_{\S^{n-1}}u^2\\
     &\hspace{14pt}+ 
     \int_{\S^{n-1}} (1+u)^{n-3}\frac{\nabla^2 u[\nabla u,\nabla u]}{(1+u)^2 + \abs{\nabla u}^2} \,.
\end{aligned}\end{equation}
If we were able to control appropriately the \emph{bad} cubic nonlinear term, we would deduce the desired stability (since $\int\abs{\nabla u}^2$ is comparable to the right-hand side of \cref{eq:main_statement_stability}, see \cref{lem:nablau_vs_normal}) as a consequence of the Poincaré inequality on the sphere (recall that, as in Fuglede's argument, we may assume that $u$ is orthogonal to the first $n$ eigenfunctions of the Laplace operator).
Controlling the cubic nonlinear term is nontrivial and, as one may expect, to control it we shall exploit the mean-convexity $H\ge 0$, which is equivalent to (see \cref{eq:mean_curvature_formula})
\begin{equation*}
    \div\bigg(
    \frac{1+u}{\sqrt{(1+u)^2 + \abs{\nabla u}^2}
    }\, \nabla u \bigg) \le n-1+\omega(\eps) \,.
\end{equation*}
For the sake of clarity let us simplify the situation, we assume that the cubic term we want to control is
\begin{equation*}
    \int_{\S^{n-1}} \nabla^2 u[\nabla u,\nabla u]
\end{equation*}
and the constraint is
\begin{equation*}
    \lapl u \le n-1 \,.
\end{equation*}

Integrating by parts and using the constraint we deduce
\begin{equation*}
    \int_{\S^{n-1}} \nabla^2 u[\nabla u,\nabla u] 
    =
    \frac12 \int_{\S^{n-1}} \nabla u\cdot \nabla(\abs{\nabla u}^2)
    =
    -\frac12 \int_{\S^{n-1}} \lapl u \abs{\nabla u}^2
    \ge
    -\frac{n-1}2 \int_{\S^{n-1}} \abs{\nabla u}^2 \,.
\end{equation*}
If we plug this inequality into \cref{eq:intro_with_cubicterm} we get
\begin{equation*}\begin{aligned}
    \int_{\partial K} H - (n-1)\Haus^{n-1}(\S^{n-1})
    &\ge
     \Big(\frac{n-3}2-\omega(\eps)\Big)\int_{\S^{n-1}}\abs{\nabla u}^2  - (n-1)(n-2)\int_{\S^{n-1}}u^2 \,.
\end{aligned}\end{equation*}
Unfortunately, this latter estimate is not useful as the right-hand side can be negative. Indeed, since we can only assume that $u$ is orthogonal to the first $n$ eigenfunctions of the Laplacian, we have \emph{only} $\int \abs{\nabla u}^2 \ge 2n \int \abs{u}^2$, which does not guarantee that the right-hand side of the last estimate is nonnegative. Thus, in order to prove the stability (or even just the validity of the inequality), a sharper argument is necessary.

To overcome this issue, we decompose $u=u_1+u_2$ where $u_1$ is the low-frequency component of $u$ and $u_2$ is the high-frequency component (or equivalently $u_1$ is the projection of $u$ on the subspace of the eigenfunctions of $-\lapl$ with eigenvalues smaller than $\lambda=\lambda(n)$). Since we can prove that \emph{only $u_2$ matters in the cubic term}, we deduce the following improved estimate for the cubic term
\begin{align*}
    \int_{\S^{n-1}} \nabla^2 u[\nabla u,\nabla u]
    &=
    \int_{\S^{n-1}} \nabla^2 u_2[\nabla u_2,\nabla u_2] + \omega(\eps)\int_{\S^{n-1}}\abs{\nabla u}^2 \\
    &=
    -\int_{\S^{n-1}} \lapl u_2\abs{\nabla u_2}^2 + \omega(\eps)\int_{\S^{n-1}}\abs{\nabla u}^2 \\
    &\ge -\frac{n-1}2  \int_{\S^{n-1}}\abs{\nabla u_2}^2 
    + \omega(\eps)\int_{\S^{n-1}}\abs{\nabla u}^2
    \,.
\end{align*}
Since $u_2$ enjoys a very strong Poincaré inequality (due to the fact that we can choose $\lambda\gg 1$), coupling \cref{eq:intro_with_cubicterm} with the latter estimate one obtains
\begin{align*}
    \int_{\partial K} H - (n-1)\Haus^{n-1}(\S^{n-1})
    \ge
    c(n) \int_{\S^{n-1}} \abs{\nabla u}^2 \,,
\end{align*}
which is equivalent to \cref{eq:main_statement_stability} (if $H\ge 0$).

\subsection{Structure of the paper}
After introducing the notation, in \cref{sec:computing_formulas} we compute various geometric quantities and we obtain formulas that will be useful later on.
The core of our work is in \cref{sec:cubic_term,sec:main_proofs}. In \cref{sec:cubic_term} we develop some techniques to control the \emph{bad cubic term} mentioned in the outline of the proof, while in \cref{sec:main_proofs} we prove the theorems stated in the introduction, namely \cref{thm:main_statement_nuclear,thm:main_statement_axial,thm:main_statement_stability,thm:main_statement_almost_sharp}.
In the last section, we isolate the obstruction that prevents us from proving the \emph{optimal statement} (as explained in the introduction) and we state a conjecture that might be of independent interest.

Finally, in \cref{app:counterexample} we show, constructing a counterexample, that \cref{eq:minkowski,eq:minkowski_vol} are false for perturbations of a ball if we were to replace $H^+$ with $H$.

\subsection{Acknowledgements}
The author is very grateful to A. Figalli for his guidance on the topic of this paper.
The author has received funding from the European Research Council under the Grant Agreement No. 721675 ``Regularity and Stability in Partial Differential Equations (RSPDE)''.

\section{Notation}\label{sec:notation}
Given two expressions $A, B$; the notation $A\lesssim B$ is equivalent to $A\le c(n)B$ where $c(n)>0$ is a constant that depends only on the dimension $n$ of the ambient $\R^n$.
Given an expression $A$, we write $B=\bigo(A)$ if $\abs{B}\lesssim A$.

We denote with $\omega(\eps)$ any expression that goes to $0$ as $\eps\to 0$. Notice that $\omega(\eps)$ may be much bigger than $\eps$. Moreover, $\omega(\eps)$ may be both positive and negative. If there is more than one occurrence of $\omega(\eps)$ in a single expression, they do not necessarily represent the same value (e.g., the expression $\omega(\eps)-\omega(\eps)$ is not necessarily null).

Given a domain $K$, the normal to the boundary is denoted by $\nu$, the second fundamental form of the boundary is denoted by $\II$, and the mean curvature of the boundary $\partial K$ (that is the trace of $\II$) is denoted by $H$.

All the integrals on the sphere $\S^{n-1}$ or on the boundary $\partial K$ are with respect to the $(n-1)$-Hausdorff measure $\Haus^{n-1}$; we will often drop the notation $\de\Haus^{n-1}$. The Lebesgue measure of $K$ is denoted by $\abs{K}$, while its perimeter (i.e., the $(n-1)$-Hausdorff measure of the boundary $\partial K$) is denoted by $\Per(K)$.

Given a real number $t\in\R$, we define $t^+=\max(0, t)$ and $t^-=-\min(0, t)$, so that $t=t^+-t^-$.

\section{Computing some geometric quantities}\label{sec:computing_formulas}
In the next lemma we collect a number of formulas which will be useful later on. Similar computations can be found in \cite[Section 2]{Urbas1990}.

\begin{lemma}\label{lem:formulas}
    Let $K\subseteq\R^n$ be a star-shaped domain with smooth boundary. Let $u:\S^{n-1}\to(-1,\infty)$ be the function such that $(1+u(x))x\in\partial K$ for all $x\in\S^{n-1}$.
    Let us define $T(x)\defeq (1+u(x))x$, so that $\partial K = T(\S^{n-1})$.
    Given $y\in\partial K$, let $\II_y$ be the second fundamental form of $\partial K$ at $y$ and let $H_y$ be the mean curvature of $\partial K$ at $y$. For the computations, it is useful to define $v(x)\defeq \frac{\nabla u(x)}{1+u(x)}$ for all $x\in\S^{n-1}$.
    
    We have the following formulas:\footnote{In the formula for $\II$, the operator $\nabla^2 u$ is extended as $0$ on the span of the vector $x$ (i.e., $\nabla^2 u[x, X] = \nabla^2 u[X, x]=0$ for any vector $X$).}
    \begin{align}
        &\abs{K} = \int_{\S^{n-1}} \frac{(1+u)^n}{n} \,, \label{eq:volume_formula}\\
        &\Per(K) = \int_{\S^{n-1}} (1+u)^{n-1}\sqrt{1+\abs{v}^2} \,, \label{eq:perimeter_formula}\\
        &\II_{T(x)} = \frac{1}{(1+u)\sqrt{1+\abs{v}^2}}\Big(\id - \frac{\nabla^2 u}{1+u} + v\otimes v\Big) \,, \label{eq:second_fund_formula}\\
        &H_{T(x)} = \frac{1}{(1+u)\sqrt{1+\abs{v}^2}}\bigg(n-1-\frac{\lapl u}{1+u} + \frac{\abs{v}^2}{1+\abs{v}^2} + \frac{\nabla^2 u[\nabla u,\nabla u]}{(1+u)^3(1+\abs{v}^2)}\bigg) \label{eq:mean_curvature_formula} \\
        &\hspace{27pt}=
        \frac{n-1 + \abs{v}^2}{(1+u)\sqrt{1+\abs{v}^2}} - (1+u)^{-2}\div\bigg(\frac{\nabla u}{\sqrt{1+\abs{v}^2}}\bigg)
        \,.\label{eq:mean_curvature_divergence}
    \end{align}
    If we assume furthermore that $\norm{u}_{C^1} < \eps$, then we have
    \begin{equation}\label{eq:integral_H_formula}\begin{aligned}
        &\int_{\partial K} H - (n-1)\Haus^{n-1}(\S^{n-1})
        =
        (n-1)(n-2)\int_{\S^{n-1}} u + \frac{(n-1)(n-2)(n-3)}{2}\int_{\S^{n-1}} u^2 \\
        &\hspace{142pt}+ (n-2)\int_{\S^{n-1}}\abs{\nabla u}^2
        + \int_{\S^{n-1}} (1+u)^{n-3}\frac{\nabla^2 u[\nabla u,\nabla u]}{(1+u)^2 + \abs{\nabla u}^2}
        \\
        &\hspace{142pt}+\omega(\eps)\int_{\S^{n-1}} u^2 + \abs{\nabla u}^2 \,,
    \end{aligned}\end{equation}
    where $\omega(\eps)$ denotes a quantity that goes to $0$ as $\eps\to0$.
\end{lemma}
\begin{proof}
    The expression \cref{eq:volume_formula} is standard (it follows by Fubini's theorem and the change of variable formula).
    
    The differential of $T$ is
    \begin{equation*}
        \de T_x = (1+u)(\id + x\otimes v)\,,
    \end{equation*}
    and its Jacobian satisfies
    \begin{equation}\label{eq:jacobian}
        J(T_x) = (1+u)^{n-1}\sqrt{1+\abs{v}^2} \,.
    \end{equation}
    The expression \cref{eq:perimeter_formula} is a direct consequence of \cref{eq:jacobian}.
    
    Let $\nu:\partial K\to\R^n$ be the outer normal to $\partial K$.
    It holds
    \begin{equation}\label{eq:normal_formula}
        \nu(T(x)) = \frac{x-v}{\sqrt{1+\abs{v}^2}} \,.
    \end{equation}
    
    Let $X, Y\in\chi(\S^{n-1})$ be two vector fields on the sphere. 
    Given another vector field $Z:\partial K\to\R^n$ (not necessarily tangent to $\partial K$), it holds the identity
    \begin{equation*}
        \nabla^{\R^n}_{\de T[X]} Z = \nabla^{\R^n}_X (Z\circ T)
    \end{equation*}
    and choosing $Z=\nu$ we obtain
    \begin{equation}\label{eq:loc_id5543}
        \nabla^{\R^n}_{\de T[X]} \nu 
        = 
        \nabla^{\R^n}_X\bigg(\frac{x-v}{\sqrt{1+\abs{v}^2}}\bigg) \,.
    \end{equation}
    Applying \cref{eq:loc_id5543}, we have
    \begin{align*}
        \II_{T(x)}(\de T_x[X], \de T_x[Y])
        &=
        \scalprod{\nabla^{\R^{n}}_{\de T_x[X]} \nu, \de T_x[Y]}
        =
        \frac1{\sqrt{1+\abs{v}^2}}
        \scalprod{\nabla^{\R^n}_X(x-v), \de T_x[Y]} \\
        &=
        \frac{1+u}{\sqrt{1+\abs{v}^2}}
        \scalprod{X-\frac{\nabla^{\R^n}_X \nabla u}{1+u} + \scalprod{v, X}v, Y+\scalprod{v, Y}x} \\
        &=
        \frac{1+u}{\sqrt{1+\abs{v}^2}}
        \left(\scalprod{X, Y} - \frac{\nabla^2 u[X, Y]}{1+u} + 2\scalprod{v, X}\scalprod{v, Y}\right) \,.
    \end{align*}
    Thanks to the previous formula, it is not hard to check that \cref{eq:second_fund_formula} holds.

    Since the mean curvature is the trace of the second fundamental form, we deduce
    \begin{equation}\label{eq:mean_curvature}
    \begin{aligned}
        H_{T(x)} 
        &= \tr^{\R^n}(\II_{T(x)}) - \II_{T(x)}(\nu, \nu) \\
        &=
        \frac1{(1+u)\sqrt{1+\abs{v}^2}}
        \left(
        n-\frac{\lapl u}{1+u} + \abs{v}^2 - 1 + \frac{\nabla^2 u[\nabla u, \nabla u]}{(1+u)^3(1+\abs{v}^2)}-\frac{\abs{v}^4}{1+\abs{v}^2}
        \right) \\
        &= \frac1{(1+u)\sqrt{1+\abs{v}^2}}
        \left(
        n-1-\frac{\lapl u}{1+u} + \frac{\abs{v}^2}{1+\abs{v}^2} + \frac{\nabla^2 u[\nabla u, \nabla u]}{(1+u)^3(1+\abs{v}^2)}
        \right) \,,    
    \end{aligned}
    \end{equation}
    which coincides with \cref{eq:mean_curvature_formula}. The formula \cref{eq:mean_curvature_divergence} follows from \cref{eq:mean_curvature_formula} thanks to the identity
    \begin{equation*}
        \div\bigg(\frac{\nabla u}{\sqrt{1+\abs{v}^2}}\bigg)
        =
        \frac{1}{\sqrt{1+\abs{v}^2}}\Big(
        \lapl u - \frac{\nabla^2 u[\nabla u,\nabla u]}{(1+u)^2 + \abs{\nabla u}^2}
        + \frac{(1+u)\abs{v}^4}{1+\abs{v}^2}\Big)
    \end{equation*}

    Joining \cref{eq:jacobian,eq:mean_curvature}, we get
    \begin{align*}
        \int_{\partial K} H =
        \int_{\S^{n-1}} (1+u)^{n-2} 
        \left(
        n-1-\frac{\lapl u}{1+u} + \frac{\abs{v}^2}{1+\abs{v}^2} + \frac{\nabla^2 u[\nabla u, \nabla u]}{(1+u)^3(1+\abs{v}^2)}
        \right) \,.
    \end{align*}
    Integrating by parts and exploiting $\norm{u}_{C^1}<\eps$, we can continue the chain of equalities with
    \begin{align*}
        \hspace{51pt}=&
        \int_{\S^{n-1}}(n-1)\big(1+(n-2)u + \frac{(n-2)(n-3)}{2}u^2\big) 
        + (n-3)\abs{\nabla u}^2 + \abs{\nabla u}^2 \\
        &\quad\ + (1+u)^{n-3}\frac{\nabla^2 u[\nabla u, \nabla u]}{(1+u)^2 + \abs{\nabla u}^2} + \omega(\eps)(u^2 + \abs{\nabla u}^2)\,,
    \end{align*}
    which is equivalent to \cref{eq:integral_H_formula}.
\end{proof}

In the following lemma we show that the deviation of the normal $\nu(y)$ to $\partial K$ from $\frac{y}{\abs{y}}$ is comparable to $\abs{\nabla u(x)}$ where $u$ is the \emph{profile function} (i.e. $(1+u(x))x\in\partial K$ for all $x\in\S^{n-1}$) of the star-shaped domain $K$.
\begin{lemma}\label{lem:nablau_vs_normal}
    There is a universal constant $\eps_0>0$ such that the following statement holds.
    
    Let $K\subseteq\R^n$ be a star-shaped domain, let $u:\S^{n-1}\to(-1,\infty)$ be the function such that $T(x)\defeq (1+u(x))x\in\partial K$ for any $x\in\S^{n-1}$. Let $\nu:\partial K\to\S^{n-1}$ be the normal to $\partial K$.
    \begin{enumerate}[label=(\arabic*)]
        \item Given $y=T(x)$, if $\abs{\nu(y)-\frac y{\abs{y}}} < \eps_0$, then \label{it:good_normal_implies_good_gradient}
        \begin{equation*}
            \frac{\abs{\nabla u(x)}}{1+u(x)} \lesssim \abs*{\nu(y)-\frac y{\abs{y}}} \,.
        \end{equation*}
        \item Given $y=T(x)$, it holds\label{it:good_gradient_implies_good_normal}
        \begin{equation*}
            \abs*{\nu(y)-\frac y{\abs{y}}} \lesssim \frac{\abs{\nabla u(x)}}{1+u(x)} \,.
        \end{equation*}
        \item If $\frac{\abs{\nabla u(x)}}{1+u(x)}<\eps_0$ for all $x\in\S^{n-1}$, then\label{it:good_gradient_implies_good_function}
        \begin{equation*}
            \frac{\max 1+u}{\min 1+u} -1 \lesssim \eps_0 \,.
        \end{equation*}
        \item If $\abs{\nu(y)-\frac y{\abs{y}}} < \eps_0$ for all $y\in\partial K$, then\label{it:int_nabla_eq_int_norm}
        \begin{equation*}
            \int_{\S^{n-1}}\frac{\abs{\nabla u}^2}{(1+u)^2}
            \lesssim
            \fint_{\partial K} \abs*{\nu(y)-\frac{y}{\abs{y}}}^2
            \lesssim
            \int_{\S^{n-1}}\frac{\abs{\nabla u}^2}{(1+u)^2} \,.
        \end{equation*}
    \end{enumerate}
\end{lemma}
\begin{proof}
    The statements \cref{it:good_normal_implies_good_gradient,it:good_gradient_implies_good_normal} follow from the formula (recall \cref{eq:normal_formula})
    \begin{equation*}
        \abs{\nu(T(x))-x}^2 = \frac{\abs{v}^2}{1+\abs{v}^2} + \frac{\abs{v}^4}{(1+\abs{v}^2)(1+\sqrt{1+\abs{v}^2})^2} \,,
    \end{equation*}
    where $v\defeq \frac{\nabla u}{1+u}$.
    
    To prove \cref{it:good_gradient_implies_good_function}, let us argue as follows. Let $R=\max(1+u)$. By assumption, we have $\abs{\nabla (\frac uR)}<\eps_0$ and therefore (since the sphere $\S^{n-1}$ is bounded) we deduce $\max \frac uR - \min \frac uR \lesssim \eps_0$. Since $\max u = R-1$, the last estimate is equivalent to the desired statement.
    
    Finally, let us prove \cref{it:int_nabla_eq_int_norm}. In the following lines we use the notation $A\approx B$ as a shorthand for $A\lesssim B$ and $B\lesssim A$.
    Up to suitably reducing the value of $\eps_0$, thanks to \cref{it:good_normal_implies_good_gradient,it:good_gradient_implies_good_normal,it:good_gradient_implies_good_function}, we may assume that $\abs{\frac{1+u}R-1}\ll 1$ and
    \begin{equation*}
        \abs*{\nu(y)-\frac y{\abs{y}}} \approx \frac{\abs{\nabla u(x)}}{1+u(x)} \,.
    \end{equation*}
    Let $R_0>0$ be the radius such that $\Haus^{n-1}(\S^{n-1})R_0^{n-1} = \Per(K)$. As a consequence of $\abs{\frac{1+u}R-1}\ll 1$, we have $\abs{\frac{1+u}{R_0}-1}\ll 1$.
    Using \cref{eq:jacobian}, we obtain (recall that $x=\frac{T(x)}{\abs{T(x)}}$)
    \begin{align*}
        \fint_{\partial K} \abs*{\nu(y)-\frac{y}{\abs{y}}}^2
        &=
        \frac1{\Haus^{n-1}(\S^{n-1})}
        \int_{\S^{n-1}} 
        \abs*{\nu(T(x))-x}^2
        \Big(\frac{1+u}{R_0}\Big)^{n-1}
        \sqrt{1+\abs*{\frac{\nabla u}{1+u}}^2}
        \\
        &\approx
        \int_{\S^{n-1}} \frac{\abs{\nabla u}^2}{(1+u)^2} \,.
    \end{align*}
\end{proof}

\section{Controlling the cubic term}\label{sec:cubic_term}
In this section we develop some techniques to control the cubic term of \cref{eq:integral_H_formula} (as explained in the outline of the proof in the introduction, this is a crucial step of the proof). Notice that such term is expected to be small (as it is cubic and everything else is quadratic) but, on the other hand, it contains the full Hessian (over which we do not have much control).

The functions $f,g$ and the parameter $\lambda$ that appear in the following statements are assumed to be fixed, hence any constant that depends on them can be absorbed in the notation $\omega(\eps)$.

\begin{lemma}\label{lem:freq_decomposition}
    Let $0\in U\subseteq\R$ be an open set and let $f,g:U\times U\to(0,\infty)$ be two positive $C^1$ functions such that $f(0,0)=g(0,0)=1$. Assume that $B(0, \bar r)\subseteq U$ for a certain $\bar r>0$.
    
    Let $u:\S^{n-1}\to(-1,\infty)$ be a function such that $\norm{u}_{C^1}\le\eps$, with $\eps < \bar r$.
    Decompose the function $u$ as $u=u_1+u_2$, where $u_1, u_2$ belong to the subspaces generated by the eigenfunctions of $-\lapl$ with eigenvalues respectively smaller and larger than a fixed $\lambda > 0$.
    
    It holds
    \begin{align*}
        \int_{\S^{n-1}}f(u,\abs{\nabla u}^2)\nabla^2 u[\nabla u,\nabla u]
        \ge
        &-\Big(\frac12+\omega(\eps)\Big)\int_{\S^{n-1}}\div(g(u,\abs{\nabla u}^2)\nabla u)\abs{\nabla u_2}^2 \\
        &+ \omega(\eps)\int_{\S^{n-1}}([\div(g(u,\abs{\nabla u}^2)\nabla u)]^+ + 1)\abs{\nabla u}^2\,.
    \end{align*}
\end{lemma}
\begin{proof}
    First of all, let us remark that $\abs{\nabla u_1}, \abs{\nabla u_2}$ can be estimated with $\omega(\eps)$. 
    Also $\norm{\nabla^2 u_1}=\omega(\eps)$ since $\norm{\nabla^2 u_1}_\infty\lesssim \norm{\nabla u_1}_{\infty}$ because all norms are equivalent on a finite-dimensional space (and $u_1$ lives in the finite-dimensional space generated by eigenfunctions of $-\lapl$ with eigenvalues $< \lambda$).
    Loosely speaking, our only enemy is $\nabla^2 u_2$. 
    
    Let $h:U\times U\to\R$ be the function such that $f(s,t)=g(s,t)h(s,t)$. The function $h$ is $C^1$ and $h(0,0)=1$.
    Let $H(s,t)\defeq\int_0^t h(s,t')\de t'$. 
    Notice that $\partial_s H(u,\abs{\nabla u}^2) = \omega(\eps)$, $\partial_t H(s, t) = h(s,t)$ and, for a certain value $0<q=\omega(\eps)$, we have
    \begin{equation}\label{eq:elem_ineq_on_g}
        (1-q)\abs{\nabla u}^2 \le H(u, \abs{\nabla u}^2)
        \le (1+q)\abs{\nabla u}^2 \,.
    \end{equation}
    By definition of $h$ and $H$, we have
    \begin{equation}\label{eq:tmp52467}\begin{aligned}
        \int_{\S^{n-1}}f(u,\abs{\nabla u}^2)\nabla^2 u[\nabla u, \nabla u]
        &=
        \int_{\S^{n-1}}g(u,\abs{\nabla u}^2)h(u,\abs{\nabla u}^2)\nabla^2 u[\nabla u, \nabla u]\\
        &=
        \frac12\int_{\S^{n-1}}g(u,\abs{\nabla u}^2)\nabla u\cdot\nabla\big(H(u,\abs{\nabla u}^2)\big)
        +\omega(\eps)\int_{\S^{n-1}}\abs{\nabla u}^2
        \,.
    \end{aligned}\end{equation}
    Integrating by parts \cref{eq:tmp52467} and applying \cref{eq:elem_ineq_on_g}, we get
    \begin{equation}\label{eq:tmp52468}\begin{aligned}
        \int_{\S^{n-1}}&g(u,\abs{\nabla u}^2)\nabla u\cdot\nabla\big(H(u,\abs{\nabla u}^2)\big)=
        -\int_{\S^{n-1}}\div(g\, \nabla u) H(u, \abs{\nabla u}^2) \\
        &= 
        -(1-q)\int_{\S^{n-1}}\div(g\, \nabla u) \abs{\nabla u}^2
        +\int_{\S^{n-1}}\div(g\,\nabla u) \Big((1-q)\abs{\nabla u}^2 - H(u,\abs{\nabla u}^2)\Big) \\
        &\ge 
        -(1-q)\int_{\S^{n-1}}\div(g\, \nabla u) \abs{\nabla u}^2
        -2q\int_{\S^{n-1}}\div(g\,\nabla u)^+\abs{\nabla u}^2 \\
        &=
        2(1+\omega(\eps))\int_{\S^{n-1}}g\, \nabla^2 u[\nabla u, \nabla u]
        +\omega(\eps)\int_{\S^{n-1}}\div(g\,\nabla u)^+\abs{\nabla u}^2 \,.        
    \end{aligned}\end{equation}
    The second term is fine (as it appears in the left-hand side of the sought inequality), hence we focus on the first term.
    Let $G(s,t)\defeq\int_0^t g(s,t')\de t'$. 
    Notice that $\partial_s G(u,\abs{\nabla u}^2) = \omega(\eps)$, $\partial_t G(s, t) = g(s,t)$ and $G(u,\abs{\nabla u}^2) = (1+\omega(\eps))\abs{\nabla u}^2$.
    
    It holds
    \begin{equation*}
        \nabla\Big(G(u, \abs{\nabla u}^2)\Big) 
        = 2g(u, \abs{\nabla u}^2) \nabla^2 u \cdot \nabla u + \omega(\eps)\nabla u
        \,.
    \end{equation*}
    Decomposing $u=u_1+u_2$ (recalling that $\norm{\nabla^2 u_1}=\omega(\eps)$) and integrating by parts, we have
    \begin{equation}\label{eq:tmp52469}\begin{aligned}
        2\int_{\S^{n-1}}&g\, \nabla^2 u[\nabla u, \nabla u]
        =
        2\int_{\S^{n-1}}g\, \nabla^2 u[\nabla u, \nabla u_1]
        +
        2\int_{\S^{n-1}}g\, \nabla^2 u_2[\nabla u, \nabla u_2]
        +\omega(\eps)\int_{\S^{n-1}} \abs{\nabla u}^2 \\
        &=
        \int_{\S^{n-1}}\nabla (G(u, \abs{\nabla u}^2))\cdot \nabla u_1
        +
        \int_{\S^{n-1}}g\, \nabla u\cdot\nabla(\abs{\nabla u_2}^2)
        +\omega(\eps)\int_{\S^{n-1}} \abs{\nabla u}^2 \\
        &=
        -\int_{\S^{n-1}}G(u, \abs{\nabla u}^2)\cdot \lapl u_1
        -\int_{\S^{n-1}}\div(g\, \nabla u)\abs{\nabla u_2}^2
        +\omega(\eps)\int_{\S^{n-1}} \abs{\nabla u}^2 \\
        &=
        -\int_{\S^{n-1}}\div(g\, \nabla u)\abs{\nabla u_2}^2
        +\omega(\eps)\int_{\S^{n-1}} \abs{\nabla u}^2  \,.
    \end{aligned}\end{equation}
    The statement follows from \cref{eq:tmp52467,eq:tmp52468,eq:tmp52469}.
\end{proof}

\begin{lemma}\label{lem:int_parts_interpol}
    Let $u:\S^{n-1}\to(-1,\infty)$ be a function such that $\norm{u}_{C^1}\le\eps$.
    Let $\lambda_1^{\nabla^2 u}\le \lambda_2^{\nabla^2 u} \le \cdots \le \lambda_{n-1}^{\nabla^2 u}$ be the $n-1$ eigenvalues of $\nabla^2 u$ (the eigenvalues depend implicitly on $x\in\S^{n-1}$).
    
    We have
    \begin{equation*}
        \int_{\S^{n-1}} \nabla^2 u[\nabla u,\nabla u]
        \ge
        -\frac13\int_{\S^{n-1}} \big(\lambda_2^{\nabla^2 u} + \lambda_3^{\nabla^2 u} + \cdots + \lambda_{n-1}^{\nabla^2 u}\big)\abs{\nabla u}^2 \,.
    \end{equation*}
\end{lemma}
\begin{proof}
    Integrating by parts, we have
    \begin{equation*}
        \int_{\S^{n-1}} \nabla^2 u[\nabla u, \nabla u]
        =
        \frac12\int_{\S^{n-1}}\nabla u\cdot \nabla\big(\abs{\nabla u}^2\big)
        =
        -\frac12\int_{\S^{n-1}}\lapl u \abs{\nabla u}^2 \,.
    \end{equation*}
    Applying the latter identity we obtain
    \begin{align*}
        \int_{\S^{n-1}} \nabla^2 u[\nabla u, \nabla u]
        &=
        \frac13\int_{\S^{n-1}} \nabla^2 u[\nabla u, \nabla u]
        -\frac13\int_{\S^{n-1}}\lapl u \abs{\nabla u}^2 
        =
        \frac13\int_{\S^{n-1}} \nabla^2 u[\nabla u, \nabla u]-\lapl u \abs{\nabla u}^2 \\
        &\ge
        \frac13\int_{\S^{n-1}} \lambda_1^{\nabla^2 u}\abs{\nabla u}^2-(\lambda_1^{\nabla^2 u}+\lambda_2^{\nabla^2 u}+\cdots+\lambda_{n-1}^{\nabla^2 u}) \abs{\nabla u}^2 \\
        &=
        -\frac13\int_{\S^{n-1}} (\lambda_2^{\nabla^2 u}+\lambda_2^{\nabla^2 u}+\cdots+\lambda_{n-1}^{\nabla^2 u}) \abs{\nabla u}^2
        \,.
    \end{align*}
\end{proof}

\begin{remark}
    Let us briefly comment on the two lemmas (we implicitly refer to the outline of the proof given in the introduction):
    \begin{itemize}
        \item The first lemma is effective when we have an a priori bound only on a pseudo-Laplacian of $u$ (i.e., $\div(g\nabla u)$ for a certain $g=g(u,\abs{\nabla u}^2)$) and not on its full Hessian.
        It comes handy for two orthogonal reasons. First, it allows us to change the nonlinearity as we desire (notice that in the left-hand side we have $f$ while in the right-hand side we have $g$). Furthermore, we can control from below the left-hand side using only the high-frequency component of the gradient (that is $\nabla u_2$, which appears on the right-hand side). Being able to use only the high-frequency component is fundamental as such component satisfies a much stronger Poincaré inequality.
        \item The second lemma is effective when we have an a priori bound from above on the full Hessian of $u$.
        Through a tricky linear combination of integrations by parts, we manage to control from below $\int \nabla^2 u[\nabla u, \nabla u]$ with $-\frac{n-2}3\int \abs{\nabla u}^2$ if $\nabla^2 u\le \id_{n-1}$ . This estimate is an improvement over the more natural $\int \nabla^2 u[\nabla u, \nabla u] = -\frac12\int \lapl u\abs{\nabla u}^2 \ge -\frac{n-1}2\int \abs{\nabla u}^2$ (under the assumption $\nabla^2 u\le \id_{n-1}$) which would not be sufficient for our purposes.
    \end{itemize}
\end{remark}

\section{Proofs of the main theorems}\label{sec:main_proofs}
\begin{proof}[Proof of \cref{thm:main_statement_nuclear}]
    Let $K$ be a $C^1$ $\eps_1$-perturbation of a ball with $\eps_1>0$. We show that the statement holds for $K$ if $\eps_1$ is sufficiently small.
    
    Without loss of generality we may assume that $\Per(K)=\Per(B_1)$, the barycenter of $K$ is the origin and it holds\footnote{The assumption that $K$ is a $C^1$ $\eps_1$-perturbation of a ball does not guarantee that we can choose $\bar x=0_{\R^n}$ in \cref{eq:def_eps_perturbation}. On the other hand, since the barycenter of $K$ is the origin, it is not hard to check that $K$ is a $C^1$ $\omega(\eps_1)$-perturbation of a ball with $\bar x=0_{\R^n}$.}
    \begin{equation*}
        \abs*{\nu(x) - \frac{x}{\abs{x}}} = \omega(\eps_1)
    \end{equation*}
    for all $x\in\partial K$. In particular $K$ is star-shaped.
    
    Let $u:\S^{n-1}\to(-1,\infty)$ be the function such that $(1+u(x))x\in\partial K$ for all $x\in\S^{n-1}$.
    Thanks to \cref{lem:nablau_vs_normal}, we have that $\eps\defeq\norm{u}_{C^1} = \omega(\eps_1)$.

    The condition on the perimeter (recall \cref{eq:perimeter_formula}) $\int_{\S^{n-1}}1 =\Per(K)= \int (1+u)^{n-1}\sqrt{1+\abs{v}^2}$ implies
    \begin{equation} \label{eq:perimeter_constraint}
        \int_{\S^{n-1}} u = -\Big(\frac{n-2}2+\omega(\eps)\Big) \int_{\S^{n-1}} u^2 - \Big(\frac1{2(n-1)}+\omega(\eps)\Big)\int_{\S^{n-1}}\abs{\nabla u}^2
            = \omega(\eps)\big(\norm{u}_{L^2} + \norm{\nabla u}_{L^2}\big)
            \,,
    \end{equation}
    and the condition on the barycenter of $K$ implies
    \begin{equation}\label{eq:barycenter_constraint}
        0 = \int_{\S^{n-1}} (1+u)^{n+1}x \implies \abs*{\int_{\S^{n-1}} u\, x} \le \omega(\eps)\norm{u}_{L^2} \,. 
    \end{equation}
    Notice that \cref{eq:perimeter_constraint,eq:barycenter_constraint} imply that $u$ is \emph{almost orthogonal} to the constant function $1$ and to the coordinate functions $x_1, x_2, \dots, x_n$ (namely $\int u\psi = \omega(\eps)\norm{u}_{L^2}$ whenever $\psi$ is one of those functions).
    Since $1$ and $x_1, \dots, x_n$ are the first eigenfunctions of the Laplace-Beltrami operator on the sphere, it follows that (recall that the first eigenvalues are $0, n-1, 2n,\dots$)
    \begin{equation}\label{eq:poincare_low}
        \int_{\S^{n-1}}\abs{\nabla u}^2 - (n-1)u^2\ge \Big(\frac{n+1}{2n}+\omega(\eps)\Big)\int_{\S^{n-1}}\abs{\nabla u}^2 \,.
    \end{equation}
    
    Plugging \cref{eq:perimeter_constraint} into \cref{eq:integral_H_formula} yields
    \begin{equation}\label{eq:tmp7531}\begin{aligned}
        \int_{\partial K} H - (n-1)\Haus^{n-1}(\S^{n-1})
        &=
        \frac{n-2}2\int_{\S^{n-1}}\abs{\nabla u}^2-\frac{(n-1)(n-2)}2\int_{\S^{n-1}} u^2 \\
        &\hspace{12pt}+ \int_{\S^{n-1}} (1+u)^{n-3}\frac{\nabla^2u[\nabla u, \nabla u]}{(1+u)^2 + \abs{\nabla u}^2}+\omega(\eps)\int_{\S^{n-1}} u^2 + \abs{\nabla u}^2 \,.
    \end{aligned}\end{equation}
    
    Applying \cref{lem:freq_decomposition} with $\lambda=\lambda(n)$ sufficiently large, with $f(s,t) = \frac{(1+s)^{n-3}}{(1+s)^2+t}$ and $g(s,t) = 1$, we get (recall that $\norm{\nabla^2 u_1}=\omega(\eps)$)
    \begin{equation}\label{eq:tmp45326}
        \int_{\S^{n-1}} (1+u)^{n-3}\frac{\nabla^2u[\nabla u, \nabla u]}{(1+u)^2 + \abs{\nabla u}^2} 
        \ge (1+\omega(\eps))\int_{\S^{n-1}}\nabla^2 u_2[\nabla u_2,\nabla u_2] 
        + \omega(\eps)\int_{\S^{n-1}}((\lapl u)^++1)\abs{\nabla u}^2 \,.
    \end{equation}
    
    Let $\lambda^{\II}_1(x)\le\dots\le\lambda^{\II}_{n-1}(x)$ be the eigenvalues of $\II_{T(x)}$, let $\lambda^{\nabla^2 u}_1(x)\le\dots\le\lambda^{\nabla^2 u}_{n-1}(x)$ be the eigenvalues of $\nabla^2 u(x)$, and let $\lambda^{\nabla^2 u_2}_1(x)\le\dots\le\lambda^{\nabla^2 u_2}_{n-1}(x)$ be the eigenvalues of $\nabla^2 u_2(x)$.
    As a consequence of $\norm{\nabla^2 u_1}=\omega(\eps)$, we have $\lambda_i^{\nabla^2 u} = \omega(\eps) + \lambda_i^{\nabla^2 u_2}$.
    Thanks to \cref{eq:second_fund_formula}, we know
    \begin{equation*}
        \lambda^{\II}_i = 1+\omega(\eps) - (1+\omega(\eps))\lambda^{\nabla^2 u}_{n-i} 
        = 1+\omega(\eps) - (1+\omega(\eps))\lambda^{\nabla^2 u_2}_{n-i}
    \end{equation*}
    and thus there is $0<q=\omega(\eps)$ such that
    \begin{equation*}
        \big(\lambda^{\II}_i\big)^- \ge \frac12\big(1+q-\lambda_{n-i}^{\nabla^2 u_2}\big)^- \,.
    \end{equation*}

    Therefore we have (recall that $\norm{\emptyparam}_1$ denotes the nuclear norm, i.e., the sum of the absolute values of the eigenvalues)
    \begin{equation}\label{eq:tmp6728}
        \norm{\II}_1-H \ge \sum_{i=1}^{n-1} \big(1 + q - \lambda^{\nabla^2 u_2}_i\big)^- \,.
    \end{equation}
    Applying \cref{lem:int_parts_interpol} to the function $u_2$, we get
    \begin{equation}\label{eq:tmp921}
        \int_{\S^{n-1}} \nabla^2 u_2[\nabla u_2,\nabla u_2]
        \ge
        -\frac13\int_{\S^{n-1}} \big(\lambda_2^{\nabla^2 u_2} + \lambda_3^{\nabla^2 u_2} + \cdots + \lambda_{n-1}^{\nabla^2 u_2}\big)\abs{\nabla u_2}^2 \,.
    \end{equation}
    
    The two estimates \cref{eq:tmp45326,eq:tmp921} imply (recall that $\norm{\nabla^2 u_1}=\omega(\eps)$)
    \begin{equation}\label{eq:tmp59908}
    \begin{aligned}
        \int_{\S^{n-1}} (1+u)^{n-3}\frac{\nabla^2u[\nabla u, \nabla u]}{(1+u)^2 + \abs{\nabla u}^2} 
        \ge
        &\sum_{i=2}^{n-1}\int_{\S^{n-1}} \Big(-\frac13+\omega(\eps)\Big)\lambda^{\nabla^2 u_2}_i\abs{\nabla u_2}^2 \\
        &+ \sum_{i=1}^{n-1}\int_{\S^{n-1}}\omega(\eps)(\lambda^{\nabla^2 u_2}_i)^+\abs{\nabla u}^2 
        + \omega(\eps)\int_{\S^{n-1}}\abs{\nabla u}^2\,.
    \end{aligned}\end{equation}
    For any $t\in\R$, the following two elementary inequalities hold (recall that $0<q=\omega(\eps)$)
    \begin{align*}
        \Big(-\frac13+\omega(\eps)\Big)t\abs{\nabla u_2}^2 + \frac14(1+q-t)^-
        &\ge
        \Big(-\frac13+\omega(\eps)\Big)\abs{\nabla u_2}^2 \,,\\
        \omega(\eps)t^+\abs{\nabla u}^2 + \frac14(1+q-t) 
        &\ge 
        \omega(\eps)\abs{\nabla u}^2 \,,
    \end{align*}
    and thus, summing \cref{eq:tmp6728,eq:tmp59908}, we deduce
    \begin{equation*}
        \int_{\S^{n-1}} (1+u)^{n-3}\frac{\nabla^2u[\nabla u, \nabla u]}{(1+u)^2 + \abs{\nabla u}^2} + \frac12(\norm{\II}_1 - H)
        \ge 
        -\frac{n-2}3\int_{\S^{n-1}}\abs{\nabla u_2}^2
        +\omega(\eps)\int_{\S^{n-1}}\abs{\nabla u}^2 \,.
    \end{equation*}
    Plugging the latter estimate in \cref{eq:tmp7531}, we obtain\footnote{We are implicitly using $\frac12\int_{\S^{n-1}}(\norm{\II}_1-H) \le \int_{\partial K} (\norm{\II}_1-H)$, which is true for $\eps$ sufficiently small (since the Jacobian \cref{eq:jacobian} is $1+\omega(\eps)$).}
    \begin{equation}\label{eq:tmp526}\begin{aligned}
        \int_{\partial K} \norm{\II}_1 - (n-1)\Haus^{n-1}(\S^{n-1})
        &\ge
        \frac{n-2}2\int_{\S^{n-1}}\abs{\nabla u_1}^2 + \frac{n-2}6\int_{\S^{n-1}}\abs{\nabla u_2}^2 \\
        &\hspace{12pt}-\frac{(n-1)(n-2)}2\int_{\S^{n-1}} u_1^2 + u_2^2 
        +\omega(\eps)\int_{\S^{n-1}} u^2 + \abs{\nabla u}^2 \,.
    \end{aligned}\end{equation}
    
    Since $u_2$ belongs to the subspace generated by eigenfunctions of the Laplace operator with eigenvalues larger than $\lamdba$, it holds
    \begin{equation}\label{eq:poincare_high}
        \int_{\S^{n-1}} \abs{\nabla u_2}^2 \ge \lambda\int_{\S^{n-1}}u_2^2 \,.
    \end{equation}
    Repeating the argument used to establish \cref{eq:poincare_low}, we can also prove
    \begin{equation}\label{eq:poincare_low_u1}
        \int_{\S^{n-1}}\abs{\nabla u_1}^2 - (n-1)u_1^2\ge \frac{n+1}{2n}\int_{\S^{n-1}}\abs{\nabla u_1}^2 +\omega(\eps)\int_{\S^{n-1}}\abs{\nabla u}^2\,.
    \end{equation}

    Using the two Poincaré inequalities \cref{eq:poincare_low_u1,eq:poincare_high}, the estimate \cref{eq:tmp526} becomes
    \begin{equation*}\begin{aligned}
        \int_{\partial K} \norm{\II}_1 - (n-1)\Haus^{n-1}(\S^{n-1})
        \ge &
        \frac{(n-2)(n+1)}{4n}\int_{\S^{n-1}}\abs{\nabla u_1}^2 
        + 
        (n-2)\Big(\frac{\lambda}6-\frac{(n-1)}2\Big)\int_{\S^{n-1}}\abs{\nabla u_2}^2 \\
        &+\omega(\eps)\int_{\S^{n-1}} u^2 + \abs{\nabla u}^2 \,.
    \end{aligned}\end{equation*}
    If $\eps$ is sufficiently small and $\lambda$ is sufficiently large, the right-hand side of the latter inequality is nonnegative and therefore the statement follows.
\end{proof}

\begin{proof}[Proof of \cref{thm:main_statement_axial}]
    We repeat the proof of \cref{thm:main_statement_nuclear} verbatim up to \cref{eq:tmp7531}.
    
    Without loss of generality, we may assume that the body $K$ is symmetric with respect to the line $\{te_1:\ t\in\R\}$. Let $V:[0,\pi]\to(-1,\infty)$ be the function such that $V(\theta)\defeq u(\cos(\theta),\sin(\theta),0,0,\dots,0)$. Notice that, since we assume that $\partial K$ is smooth, $\dot V(0)=\dot V(\pi)=0$.
    
    Whenever in a formula we have both $\theta$ and $x$, it is always assumed implicitly that $\theta=\theta(x)=\cos(x_1)$. The function $u$ is always implicitly evaluated at $x$, while the function $V$ is always implicitly evaluated at $\theta$.
    
    As a consequence of the coarea formula, we have
    \begin{equation*}
        \int_{\S^{n-1}} f(\theta)\de \Haus^{n-1}(x) = \Haus^{n-2}(\S^{n-2})\int_{0}^\pi f(\theta)\sin(\theta)^{n-2} \de\theta
    \end{equation*}
    for any continuous function $f:[0,\pi]\to\R$.
    
    Standard computations yield the following formulas\footnote{The formula for $\nabla^2u[\nabla u,\nabla u]$ can be obtained by taking two derivatives of $u$ along a geodesic going from the north pole to the south pole. 
    Since we already know how to integrate with respect to $\theta$ and we are able to compute the gradient of a function which depends only on $\theta$, the formula for the Laplacian follows from the identity $\int \lapl u\, \varphi = -\int \nabla u\cdot\nabla\phi$ for any function $\phi(x) = \Phi(\theta)$.}
    \begin{align*}
        \abs{\nabla u} = \abs{\dot V} \,,\quad
        \nabla^2 u[\nabla u,\nabla u] &= \dot V^2\ddot V \,,\quad
        \lapl u = \ddot V + (n-2)\dot V\frac{\cos(\theta)}{\sin(\theta)} \,.
    \end{align*}
    
    Using these formulas, the cubic term appearing in \cref{eq:tmp7531} becomes
    \begin{equation}\label{eq:tmp942}
        \int_{\S^{n-1}} (1+u)^{n-3}\frac{\nabla^2u[\nabla u, \nabla u]}{(1+u)^2 + \abs{\nabla u}^2}
        = \Haus^{n-2}(\S^{n-2}) \int_{0}^\pi(1+V)^{n-3}\frac{\ddot V\dot V^2}{(1+V)^2+\dot V^2}\sin(\theta)^{n-2} \de\theta\,.
    \end{equation}
    Let $g(V,\dot V)\defeq (1+V)^{n-3}\frac{\dot V^2}{(1+V)^2+\dot V^2}$ and $G(V,\dot V)\defeq\int_0^{\dot V} f(V, s)\de s$. 
    One can check that $G(V, \dot V)=(\frac13+\omega(\eps))\dot V^3$ and $\abs{\partial_1G(V,\dot V)}\lesssim \abs{\dot V}^2$. 
    Thus the expression in \cref{eq:tmp942} is equal to (up to the constant factor $\Haus^{n-2}(\S^{n-2})$)
    \begin{equation}\label{eq:tmp943}\begin{aligned}
        &= \int_{0}^\pi \partial_2 G(V,\dot V)\ddot V \sin(\theta)^{n-2}\de\theta
        = \int_{0}^\pi \Big(\frac{\de}{\de\theta}G(V,\dot V)-\partial_1G(V,\dot V)\dot V\Big)\sin(\theta)^{n-2}\de\theta \\
        &=
        -\frac{n-2}{3}\int_0^{\pi}(1+\omega(\eps))\dot V^3\cos(\theta)\sin(\theta)^{n-3}\de\theta + \omega(\eps)\int_{\S^{n-1}}\abs{\nabla u}^2 \,,
    \end{aligned}\end{equation}
    where in the last step we have integrated by parts and we have used that $\abs{\dot V}=\abs{\nabla u}=\omega(\eps)$.
    
    It remains to take care of the term $\int \dot V^3\cos(\theta)\sin(\theta)^{n-3}$. Notice that this is not trivially comparable to $\int\abs{\nabla u}^2 = \int \dot V^2\sin(\theta)^{n-2}$ because the exponents of $\dot V$ and $\sin(\theta)$ are different in the two expressions.
    
    Let $\theta_0=\theta_0(\eps)\defeq \sqrt{\eps}$, so that $\theta_0=\omega(\eps)$ and $\abs{\dot V(\theta)}=\omega(\eps)\sin(\theta)$ for each $\theta_0<\theta<\pi-\theta_0$.

    We have
    \begin{equation}\label{eq:tmp4216}\begin{aligned}
    \int_0^{\pi}&(1+\omega(\eps))\dot V^3\cos(\theta)\sin(\theta)^{n-3}\de\theta \\
    &=
    \int_0^{\theta_0}(1+\omega(\eps))\dot V^3\sin(\theta)^{n-3}\de\theta
    -\int_{\pi-\theta_0}^{\pi}(1+\omega(\eps))\dot V^3\sin(\theta)^{n-3}\de\theta
    +
    \omega(\eps)\int_{\S^{n-1}} \abs{\nabla u}^2
    \end{aligned}\end{equation}
    
    In order to proceed, we are going to show that, for an appropriate universal constant $C_0=C_0(n)$, we have
    \begin{equation}\label{eq:estimate_poles}
        \dot V(\theta) \le (1+\omega(\eps))\theta + \frac{C_0}{\theta^{n-2}}\int_{\partial K}H^-\,,
    \end{equation}
    for each $0\le \theta\le \theta_0$, where $H^-=-\min(0,H)$. 
    From \cref{eq:mean_curvature_formula} we deduce
    \begin{align*}
        H &= \frac{1}{\sqrt{(1+V)^2+\dot V^2}}
        \Big(
        (n-1)-\frac{\ddot V + (n-2)\dot V\frac{\cos(\theta)}{\sin(\theta)}}{1+V} + \frac{\dot V^2}{(1+V)^2+\dot V^2} 
        + \frac{\ddot V\dot V^2}{(1+V)((1+V)^2+\dot V^2)}
        \Big)
        \\
        &=
        n-1- \omega(\eps) - 
        (1+\omega(\eps))\Big(\ddot V + (n-2)\dot V\Big(1+\frac{\dot V^2}{(1+V)^2}\Big)\frac{\cos(\theta)}{\sin(\theta)}\Big)
    \end{align*}
    and, through suitable algebraic manipulations, we obtain
    \begin{align*}
        \frac{\de}{\de\theta}\big(\dot V\sin(\theta)^{n-2}\big) 
        &= \big(n-1+\omega(\eps)-H(1+\omega(\eps))\big)\sin(\theta)^{n-2} - (n-2)\frac{\dot V^3}{(1+V)^2}\cos(\theta)\sin(\theta)^{n-3} \\
        &\le
        (n-1+\omega(\eps)+2H^-)\sin(\theta)^{n-2}
        - (n-2)\frac{\dot V^3}{(1+V)^2}\cos(\theta)\sin(\theta)^{n-3}
        \,.
    \end{align*}
    
    Now, let us fix $0<\theta_2<\theta_0$. 
    If $\dot V(\theta_2)\le 0$, then \cref{eq:estimate_poles} is trivial for $\theta=\theta_2$. Otherwise, let $0\le \theta_1<\theta_2$ be the largest value  such that $\dot V(\theta_1)=0$ (such a value exists because $\dot V(0)=0$). Since $\dot V(\theta)\ge 0$ for any $\theta\in[\theta_1,\theta_2]$, we deduce
    \begin{equation*}
        \frac{\de}{\de\theta}\big(\dot V\sin(\theta)^{n-2}\big) \le 
        \big(n-1+\omega(\eps)+2H^-\big)\sin(\theta)^{n-2}
    \end{equation*}
    in the interval $\theta\in[\theta_1,\theta_2]$. 
    Integrating the latter inequality in the interval $[\theta_1,\theta_2]$ and exploiting $\sin(\theta)=(1+\omega(\eps))\theta$, $\cos(\theta)=1+\omega(\eps)$, we get
    \begin{equation*}
        (1+\omega(\eps))\dot V(\theta_2)\theta_2^{n-2}\le
        (1+\omega(\eps))\theta_2^{n-1}+3\int_0^{\theta_2}H^{-}\sin(\theta)^{n-2}\de\theta
    \end{equation*}
    and \cref{eq:estimate_poles} follows for $\theta=\theta_2$ (notice that $\int_{\partial K}H^- = (1+\omega(\eps))\int_{\S^{n-1}}H^-$). 
    By exploiting the symmetry $(x_1,x_2,\dots, x_n)\mapsto(-x_1,x_2,\dots,x_n)$, we have also
    \begin{equation*}\label{eq:estimate_pole_south}
        -\dot V(\theta) \le (1+\omega(\eps))(\pi-\theta) + \frac{C_0}{(\pi-\theta)^{n-2}}\int_{\partial K}H^-
    \end{equation*}
    for each $\pi-\theta_0\le \theta\le\pi$.
    
    It is elementary to check that, for any $\alpha>1$ (in the end, we will choose a value of $\alpha$ very close to $1$) there is a constant $C(\alpha)$ so that if $x,y,z\ge 0$ satisfy $x\le y+z$, then $x^3\le \alpha x^2y + C(\alpha)z^3$.
    To simplify the notation, let $Q\defeq C_0\int_{\partial K}H^-$
    Therefore, \cref{eq:estimate_poles} implies
    \begin{equation*}
        \dot V^3 \le \dot V^2(\alpha+\omega(\eps)) \theta 
        +
        C(\alpha)\min\Big\{\frac{Q}{\theta^{n-2}}, \eps\Big\}^3\,,
    \end{equation*}
    where we have used $\abs{\dot V}\le \eps$ (which holds because the domain $K$ is a $C^1$ $\eps$-perturbation of the ball).
    So, the first term in the right-hand side of \cref{eq:tmp4216} can be estimated as follows
    \begin{equation}\label{eq:tmp2411}\begin{aligned}
        \int_0^{\theta_0}(1+\omega(\eps))\dot V^3\sin(\theta)^{n-3}\de\theta 
        \le (\alpha+\omega(\eps))\int_0^{\theta_0} \dot V^2 \sin(\theta)^{n-2}
        +
        C(\alpha)\int_0^{\infty}\min\Big\{\frac{Q}{\theta^{n-2}},\eps\Big\}^3\theta^{n-3}\de\theta \,.
    \end{aligned}\end{equation} 
    Let us define the two neighborhoods of the north and south pole respectively
    \begin{align*}
        &P_{\theta_0}^+ \defeq \{x\in\S^{n-1}:\ x_1\ge\cos(\theta_0)\} \,, \\
        &P_{\theta_0}^- \defeq \{x\in\S^{n-1}:\ x_1\le-\cos(\theta_0)=\cos(\pi-\theta_0)\} \,.
    \end{align*}
    Thanks to the usual change of variable, the first term is equal to 
    \begin{equation}\label{eq:tmp2412}
        \int_0^{\theta_0} \dot V^2 \sin(\theta)^{n-2} 
        =
        \Haus^{n-2}(\S^{n-2})^{-1}\int_{P^+_{\theta_0}}\abs{\nabla u}^2 \,,
    \end{equation}
    while, for the second term, we have
    \begin{equation}\label{eq:tmp2413}
        \int_0^{\infty}\min\Big\{\frac{Q}{\theta^{n-2}},\eps\Big\}^3\theta^{n-3}\de\theta
        =
        \frac1{n-2}\Big(Q\eps^2 + Q^3\frac{\eps^2}{2Q^2}\Big) = \omega(\eps)Q
    \end{equation}
    Joining \cref{eq:tmp2411,eq:tmp2412,eq:tmp2413}, we finally obtain
    \begin{equation*}\begin{aligned}
         \int_0^{\theta_0}(1+\omega(\eps))\dot V^3\sin(\theta)^{n-3}\de\theta 
        \le (\alpha+\omega(\eps))\Haus^{n-2}(\S^{n-2})^{-1}\int_{P^+_{\theta_0}}\abs{\nabla u}^2
        +\omega(\eps)C(\alpha)\int_{\partial K}H^- \,.
    \end{aligned}\end{equation*}
    If we repeat the same argument for the range of angles $[\pi-\theta_0,\pi]$, we obtain the estimate
    \begin{equation*}\begin{aligned}
        -\int_0^{\theta_0}(1+\omega(\eps))\dot V^3\sin(\theta)^{n-3}\de\theta 
        \le (\alpha+\omega(\eps))\Haus^{n-2}(\S^{n-2})^{-1}\int_{P^-_{\theta_0}}\abs{\nabla u}^2
        +\omega(\eps)C(\alpha)\int_{\partial K}H^- \,.
    \end{aligned}\end{equation*}
    Hence, by applying the latter two estimates together with \cref{eq:tmp942,eq:tmp943,eq:tmp4216}, we get
    \begin{equation}\label{eq:tmp145}\begin{aligned}
        \int_{\S^{n-1}} (1+u)^{n-3}\frac{\nabla^2u[\nabla u, \nabla u]}{(1+u)^2 + \abs{\nabla u}^2}
        \ge &
        -\frac{n-2}3(\alpha+\omega(\eps))\int_{P_{\theta_0}^+\cup P_{\theta_0}^-} \abs{\nabla u}^2
        +\omega(\eps)C(\alpha)\int_{\partial K}H^- 
        \\
        &+ \omega(\eps)\int_{\S^{n-1}}\abs{\nabla u}^2 \,.
    \end{aligned}\end{equation}
    
    Let us fix $\lambda=\lambda(n)>0$ (any sufficiently large $\lambda$ works). 
    Let us decompose the function as $u=u_1+u_2$, where $u_1,u_2$ belong to the subspaces generated by the eigenfunctions of $-\lapl$ with eigenvalues respectively smaller and larger than $\lambda$.
    
    For any constant $C>0$, we have the elementary inequality $\abs{x+y}^2 \le (1+C)\abs{x}^2 + (1+C^{-1})\abs{y}^2$. Thus, since $\nabla u = \nabla u_1 + \nabla u_2$, it holds that
    \begin{equation}\label{eq:tmp822}
        \int_{P_{\theta_0}^+\cup P_{\theta_0}^-} \abs{\nabla u}^2 \le (1+C)\int_{P_{\theta_0}^+\cup P_{\theta_0}^-}\abs{\nabla u_1}^2
        +
        (1+C^{-1})\int_{P_{\theta_0}^+\cup P_{\theta_0}^-}\abs{\nabla u_2}^2 \,.
    \end{equation}
    Since $u_1$ belongs to a finite-dimensional subspace (generated by the eigenfunctions of $-\lapl$ with eigenvalues smaller than $\lambda$) we know that all norms of $u_1$ are equivalent and therefore
    \begin{equation}\label{eq:tmp823}
        \int_{P_{\theta_0}^+\cup P_{\theta_0}^-}\abs{\nabla u_1}^2 \le \Haus^{n-1}(P_{\theta_0}^+\cup P_{\theta_0}^-)\norm{\nabla u_1}_{\infty}^2 \lesssim \omega(\theta_0)\int_{\S^{n-1}}\abs{\nabla u_1}^2 
        \,.
    \end{equation}
    Choosing $C$ appropriately depending on $\theta_0$, it follows from \cref{eq:tmp822,eq:tmp823} (notice that $\omega(\theta_0)=\omega(\eps)$)
    \begin{equation}\label{eq:poles_estimate}
        \int_{P_{\theta_0}^+\cup P_{\theta_0}^-} \abs{\nabla u}^2 \le \int_{\S^{n-1}}\abs{\nabla u_2}^2
        + \omega(\eps)\int_{\S^{n-1}}\abs{\nabla u}^2 \,.
    \end{equation}
    
    Applying \cref{eq:poles_estimate}, \cref{eq:tmp145} becomes
    \begin{equation*}\begin{aligned}
        \int_{\S^{n-1}} (1+u)^{n-3}\frac{\nabla^2u[\nabla u, \nabla u]}{(1+u)^2 + \abs{\nabla u}^2}
        \ge &
        -\frac{n-2}3(\alpha+\omega(\eps))\int_{\S^{n-1}} \abs{\nabla u_2}^2
        +\omega(\eps)C(\alpha)\int_{\partial K}H^- 
        \\
        &+ \omega(\eps)\int_{\S^{n-1}}\abs{\nabla u}^2 \,.
    \end{aligned}\end{equation*}
    The latter estimate, coupled with \cref{eq:tmp7531}, implies
    \begin{equation*}\begin{aligned}
        \int_{\partial K} H^+ - (n-1)\Haus^{n-1}(\S^{n-1})
        &=
        \int_{\partial K} H - (n-1)\Haus^{n-1}(\S^{n-1}) + \int_{\partial K} H^- \\
        &\ge
        \frac{n-2}2\int_{\S^{n-1}}\abs{\nabla u}^2-\frac{(n-1)(n-2)}2\int_{\S^{n-1}} u^2 \\
        &-\frac{n-2}3(\alpha+\omega(\eps))\int_{\S^{n-1}}\abs{\nabla u_2}^2 +\omega(\eps)\int_{\S^{n-1}} u^2 + \abs{\nabla u}^2 \\
        &+ (1-\omega(\eps)C(\alpha))\int_{\partial K}H^-\,.
    \end{aligned}\end{equation*}
    Let us choose $\alpha\defeq \frac54$ so that $\frac{n-2}2-\frac{n-2}3(\alpha+\omega(\eps))>c_0 \defeq 1/ 13$. Moreover, since $\alpha$ is fixed, we have (for $\eps$ sufficiently small) $1-\omega(\eps)C(\alpha)>0$. Therefore we have
    \begin{align*}
        \int_{\partial K} H^+ - (n-1)\Haus^{n-1}(\S^{n-1}) 
        \ge
        &\Big(\frac{n-2}2+\omega(\eps)\Big)\int_{\S^{n-1}}\abs{\nabla u_1}^2 - \Big(\frac{(n-1)(n-2)}2+\omega(\eps)\Big)\int_{\S^{n-1}}u_1^2 \\
        &+c_0\int_{\S^{n-1}}\abs{\nabla u_2}^2 - \Big(\frac{(n-1)(n-2)}2+\omega(\eps)\Big)\int_{\S^{n-1}}u_2^2 \,.
    \end{align*}
    From here we can conclude exactly as in the proof of \cref{thm:main_statement_nuclear}.
\end{proof}

\begin{proof}[Proof of \cref{thm:main_statement_stability}]
    Let $K$ be a $C^1$ $\eps_3$-perturbation of a ball with $\eps_3>0$. We show that the statement holds for $K$ if $\eps_3$ is sufficiently small.
    
    Without loss of generality we may assume that $\abs{K}=\abs{B_1}$, the barycenter of $K$ is the origin and it holds
    \begin{equation*}
        \abs*{\nu(x) - \frac{x}{\abs{x}}} = \omega(\eps_3)
    \end{equation*}
    for all $x\in\partial K$. In particular $K$ is star-shaped.
    
    Let $u:\S^{n-1}\to(-1,\infty)$ be the function such that $(1+u(x))x\in\partial K$ for all $x\in\S^{n-1}$.
    Thanks to \cref{lem:nablau_vs_normal}, we have that $\eps\defeq\norm{u}_{C^1} = \omega(\eps_3)$.
    
    The condition on the volume and the condition on the barycenter of $K$ imply (recall \cref{eq:volume_formula})
    \begin{align}
        \int_{\S^{n-1}}\frac1n = \int_{\S^{n-1}} \frac{(1+u)^n}n &\implies \label{eq:volume_constraint}
        \int_{\S^{n-1}} u = -\Big(\frac{n-1}2+\omega(\eps)\Big) \int_{\S^{n-1}} u^2 
        = \omega(\eps)\norm{u}_{L^2}\,,\\
        0 = \int_{\S^{n-1}} (1+u)^{n+1}x &\implies \abs*{\int_{\S^{n-1}} u\, x} \le \omega(\eps)\norm{u}_{L^2} \,. \label{eq:barycenter_constraint2}
    \end{align}
    As in the proof of \cref{thm:main_statement_nuclear} (there we used \cref{eq:perimeter_constraint,eq:barycenter_constraint}), from \cref{eq:volume_constraint,eq:barycenter_constraint2} one can deduce the Poincaré inequality \cref{eq:poincare_low}.
    
    Plugging \cref{eq:volume_constraint} into \cref{eq:integral_H_formula} yields
    \begin{equation}\label{eq:tmp3462}\begin{aligned}
        \int_{\partial K} H - (n-1)\Haus^{n-1}(\S^{n-1})
        &=
        (n-2)\int_{\S^{n-1}}\abs{\nabla u}^2-(n-1)(n-2)\int_{\S^{n-1}} u^2 \\
        &+ \int_{\S^{n-1}} (1+u)^{n-3}\frac{\nabla^2u[\nabla u, \nabla u]}{(1+u)^2 + \abs{\nabla u}^2}+\omega(\eps)\int_{\S^{n-1}} u^2 + \abs{\nabla u}^2 \,.
    \end{aligned}\end{equation}
    
    Let $H^-\defeq -\min(H, 0)$ be the negative part of the mean curvature, thanks to \cref{eq:mean_curvature_divergence}, we have
    \begin{equation*}
        \div\bigg(\frac{\nabla u}{\sqrt{1+\abs{v}^2}}\bigg) 
        \le n-1+\omega(\eps) + 2 H^- \,.
    \end{equation*}
    Applying \cref{lem:freq_decomposition} with $\lambda=\lambda(n)$ sufficiently large, with $f(s,t) = \frac{(1+s)^{n-3}}{(1+s)^2+t}$ and $g(s,t) = (1+\frac{t}{(1+s)^2})^{-\frac12}$, we obtain
    \begin{equation}\label{eq:tmp13}\begin{aligned}
        \int_{\S^{n-1}} (1+u)^{n-3}\frac{\nabla^2u[\nabla u, \nabla u]}{(1+u)^2 + \abs{\nabla u}^2} 
        &\ge 
        -\frac12\int_{\S^{n-1}}(n-1+\omega(\eps)+2H^-)\abs{\nabla u_2}^2 \\
        &\hspace{12pt}
        + \omega(\eps)\int_{\S^{n-1}}(n-1+\omega(\eps) + 2H^-)\abs{\nabla u}^2 \\
        &
        =-\frac{n-1}2\int_{\S^{n-1}}\abs{\nabla u_2}^2 + \omega(\eps)\int_{\S^{n-1}}\abs{\nabla u}^2 + \omega(\eps)\int_{\S^{n-1}}H^-
    \end{aligned}\end{equation}
    and therefore we have
    \begin{align*}
        \int_{\S^{n-1}} (1+u)^{n-3}\frac{\nabla^2u[\nabla u, \nabla u]}{(1+u)^2 + \abs{\nabla u}^2} + \frac12H^-
        \ge -\frac{n-1}2\int_{\S^{n-1}}\abs{\nabla u_2}^2
        + \omega(\eps)\int_{\S^{n-1}} \abs{\nabla u}^2 \,.
    \end{align*}
    Joining the latter inequality with \cref{eq:tmp3462}, we obtain\footnote{We are implicitly using $\frac12\int_{\S^{n-1}}H^- \le \int_{\partial K} H^-$, which is true for $\eps$ sufficiently small (since the Jacobian \cref{eq:jacobian} is $1+\omega(\eps)$).}
    \begin{equation*}\begin{aligned}
        \int_{\partial K} H^+ - (n-1)\Haus^{n-1}(\S^{n-1})
        &\ge
        (n-2)\int_{\S^{n-1}}\abs{\nabla u_1}^2 
        + \frac{n-3}2\int_{\S^{n-1}}\abs{\nabla u_2}^2 \\
        &\hspace{12pt}-(n-1)(n-2)\int_{\S^{n-1}} u_1^2 + u_2^2
        +\omega(\eps)\int_{\S^{n-1}} u^2 + \abs{\nabla u}^2 \,.
    \end{aligned}\end{equation*}
    Notice that since $n\ge 4$, the coefficient in front of $\int\abs{\nabla u_2}^2$ is strictly positive. Using that $u_1$ satisfies the Poincaré inequality \cref{eq:poincare_low_u1} and $u_2$ satisfies the Poincaré inequality \cref{eq:poincare_high}, from the latter estimate we deduce
    \begin{equation*}
        \int_{\partial K} H^+\de\Haus^{n-1} - (n-1)\Haus^{n-1}(\S^{n-1})
        \gtrsim \int_{\S^{n-1}} \abs{\nabla u}^2 \,,
    \end{equation*}
    which implies (given that $\abs{K}=\abs{B_1}$)
    \begin{equation*}
        \frac{\Big(\int_{\partial K} H^+\de\Haus^{n-1}\Big)^{\frac1{n-2}}}{\abs{K}^{\frac1n}} - \frac{\Big((n-1)\Haus^{n-1}(\S^{n-1})\Big)^{\frac1{n-2}}}{\abs{B_1}^{\frac1n}}
        \gtrsim  \int_{\S^{n-1}} \abs{\nabla u}^2 \,.
    \end{equation*}
    Using \cref{lem:nablau_vs_normal} once again, we have
    \begin{equation*}
        \int_{\S^{n-1}} \abs{\nabla u}^2 \gtrsim \fint_{\partial K} \abs*{\nu(x) - \frac{x}{\abs{x}}}^2 \de\Haus^{n-1}(x) \,.
    \end{equation*}
    The last two estimates imply the desired statement.
\end{proof}

\begin{proof}[Proof of \cref{thm:main_statement_almost_sharp}]
    We repeat the proof of \cref{thm:main_statement_nuclear} verbatim up to \cref{eq:tmp7531}.
    As a consequence of \cref{eq:tmp7531}, we have
    \begin{equation*}
        \int_{\partial K} H - (n-1)\Haus^{n-1}(\S^{n-1}) = \omega(\eps) 
        + \int_{\S^{n-1}} (1+u)^{n-3}\frac{\nabla^2 u[\nabla u, \nabla u]}{(1+u)^2 + \abs{\nabla u}^2} \,.
    \end{equation*}
    As a consequence of \cref{eq:tmp13} (whose proof can be repeated verbatim), we also have (recall that $\omega(\eps)$ can be a negative quantity)
    \begin{equation*}
        \int_{\S^{n-1}} (1+u)^{n-3}\frac{\nabla^2 u[\nabla u, \nabla u]}{(1+u)^2 + \abs{\nabla u}^2} \ge \omega(\eps) + \omega(\eps)\int_{\partial K} H^- \,.
    \end{equation*}
    Therefore, we obtain
    \begin{align*}
        \int_{\partial K} H^+ - (n-1)\Haus^{n-1}(\S^{n-1}) &= 
        \Big(\int_{\partial K} H - (n-1)\Haus^{n-1}(\S^{n-1})\Big) + \int_{\partial K} H^- 
        \\
        &\ge \omega(\eps) + \omega(\eps)\int_{\partial K} H^- + \int_{\partial K} H^-
        \ge \omega(\eps) \,,
    \end{align*}
    which implies the desired statement (since we normalized $\Per(K)=\Per(B_1)$ and $\abs{\omega(\eps)}$ can be assumed to be much smaller than the constant $\delta$ appearing in the statement).
\end{proof}

\section{A conjecture related to the missing statement}\label{sec:open_conjecture}
In this work we have not been able to establish the inequality
\begin{equation}\label{eq:missing_ineq}
    \int_{\partial K} H^+\de\Haus^{n-1} \ge (n-1)\Haus^{n-1}(\S^{n-1})
\end{equation}
for a $C^1$ $\eps$-perturbation $K\subseteq\R^n$ of a ball with $\Per(K)=\Per(B_1)$. 
We expect this inequality to be true, first and foremost because it is true for outward-minimizing sets (as observed in the introduction).
In this section we want to isolate the main obstruction that prevents us from proving \cref{eq:missing_ineq} with our strategy. Our reasoning will not be entirely formal, but it shall be clear why the terms we ignore do not play an important role.

Let us consider \cref{eq:missing_ineq} under the additional assumption\footnote{Under this assumption, \cref{eq:missing_ineq} is true for $C^1$ $\eps$-perturbations of the ball as shown in \cite{GuanLi2009}.} $H\ge 0$. Hence we would like to show 
\begin{equation}\label{eq:conj_sect_starting}
    \int_{\partial K} H\de\Haus^{n-1} \ge (n-1)\Haus^{n-1}(\S^{n-1})
\end{equation}
under the constraint $H\ge 0$.
Let $u:\S^{n-1}\to(-1,\infty)$ be the function such that $(1+u(x))x\in\partial K$ for all $x\in\S^{n-1}$.
Reasoning as in the proof of \cref{thm:main_statement_nuclear}, thanks to \cref{eq:tmp7531,eq:mean_curvature_divergence}, we have that the inequality \cref{eq:conj_sect_starting} is equivalent to
\begin{equation*}
    \frac{n-2}2\int\limits_{\S^{n-1}}\abs{\nabla u}^2
    -\frac{(n-1)(n-2)}2\int\limits_{\S^{n-1}} u^2
    + \int\limits_{\S^{n-1}} (1+u)^{n-3}\frac{\nabla^2u[\nabla u, \nabla u]}{(1+u)^2 
    + \abs{\nabla u}^2}
    + \omega(\eps)\int\limits_{\S^{n-1}} u^2 + \abs{\nabla u}^2  \ge 0
\end{equation*}
under the constraint
\begin{equation*}
    \frac{n-1 + \abs{v}^2}{(1+u)\sqrt{1+\abs{v}^2}} \ge (1+u)^{-2}\div\bigg(\frac{\nabla u}{\sqrt{1+\abs{v}^2}}\bigg) \,.
\end{equation*}
Let us make some considerations:
\begin{enumerate}
    \item The term $\int_{\S^{n-1}}u^2$ is much smaller than $\int_{\S^{n-1}}\abs{\nabla u}^2$ due to the Poincaré inequality (recall that in our proofs only the high-frequency part of $u$ plays a role).
    \item The term $\omega(\eps)\int_{\S^{n-1}} u^2 + \abs{\nabla u}^2$ should be negligible due to the $\omega(\eps)$ coefficient in front.
    \item The term $(1+u)^{n-3}\frac{\nabla^2u[\nabla u,\nabla u]}{(1+u)^2+\abs{\nabla u}^2}$ can be treated like $\nabla^2 u[\nabla u,\nabla u]$.
    \item It holds $\frac{n-1 + \abs{v}^2}{(1+u)\sqrt{1+\abs{v}^2}} \approx n-1$.
    \item The term $(1+u)^{-2}\div\bigg(\frac{\nabla u}{\sqrt{1+\abs{v}^2}}\bigg)$ can be treated like $\lapl u$.
\end{enumerate}
With these considerations in mind, we see that our problem is \emph{morally equivalent} to showing
\begin{equation*}
    \frac{n-2}{2}\int_{\S^{n-1}}\abs{\nabla u}^2 + \int_{\S^{n-1}}\nabla^2 u[\nabla u,\nabla u] \ge 0
\end{equation*}
under the constraint
\begin{equation*}
    n-1 \ge \lapl u \,.
\end{equation*}
Integrating by parts, it is the same as proving
\begin{equation*}
    \int_{\S^{n-1}}\lapl u\abs{\nabla u}^2 \le (n-2)\int_{\S^{n-1}}\abs{\nabla u}^2 
    \quad\text{provided that}\quad \lapl u\le n-1 \,.
\end{equation*}
We have finally reached the obstruction: even if we assume that $\norm{\nabla u}_{\infty}\ll 1$ and we assume that $u$ has only high frequencies (as we have done in our proofs) this latter inequality remains elusive\footnote{Notice that \cref{lem:int_parts_interpol} establish such an inequality (with an even better constant) if $\lapl u\le n-1$ is replaced by the stronger assumption $\nabla^2 u\le \id$.}.
Let us state this as a conjecture. We believe that the methods necessary to prove such an inequality would be sufficient to prove \cref{eq:missing_ineq}.
\begin{conjecture}
    For $n\ge 2$, let $u:\S^{n-1}\to\R$ be a smooth function with $\lapl u\le 1$ (and, if necessary, also $\norm{\nabla u}_{\infty}\ll 1$). Then it holds
    \begin{equation*}
        \int_{\S^{n-1}} \lapl u\abs{\nabla u}^2 \le \frac{n-2}{n-1}\int_{\S^{n-1}}\abs{\nabla u}^2 \,.
    \end{equation*}
\end{conjecture}
\begin{remark}
    If the constant $\frac{n-2}{n-1}$ is replaced by $1$, the conjecture becomes trivial.
    We do not believe that the constant $\frac{n-2}{n-1}$ is the optimal one (but it is the constant that comes up naturally in our work). When $n=2$, the conjecture follows from the elementary identity $\lapl u\abs{\nabla u}^2=\frac{\de(\abs{\nabla u}^3)}3$ (valid only if the domain is $1$-dimensional).
    
    We are not able to prove the conjecture even if $\frac{n-2}{n-1}$ is replaced by any constant strictly less than $1$. If we were able to prove the conjecture with a constant strictly below $1$, then, very likely, we would be able to show \cref{thm:main_statement_stability} also for $n=3$.
    
    For $n\ge 4$, one can find a sequence of functions $(u_k)_{k\in\N}$ such that $\lapl u_k\le 1$, $\norm{\nabla u_k}_{\infty}\to 0$ as $k\to\infty$, and
    \begin{equation*}
        \liminf_{k\to\infty}\frac{\int_{\S^{n-1}}\lapl u_k\abs{\nabla u_k}^2}{\int_{\S^{n-1}}\abs{\nabla u_k}^2} > 0 \,.
    \end{equation*}
    Hence, even under the constraint $\norm{\nabla u_k}_{\infty}\ll 1$, the desired inequality is \emph{dimensionally} sharp. Such a sequence of $u_k$ can be constructed by appropriately convolving the Green function on the sphere.
\end{remark}

\appendix
\section{Necessity of curvature bounds}\label{app:counterexample}
One might be tempted to claim that \cref{eq:minkowski} or \cref{eq:minkowski_vol} hold (possibly with a non-sharp constant) in a $C^1$ neighborhood of the ball without any curvature assumption and without replacing $H$ with $H^+$. 
This is not true (for $n\ge 3$) and, since the construction of the counterexample is not as straight-forward as one might expect, we prove it.

First we need a technical lemma.

\begin{lemma}\label{lem:radial_int_parts}
    Given $n\ge 3$, let $u:\R^{n-1}\to\R$ be a smooth radial function with compact support, so $u(x)=f(\abs{x})$ where $f:[0,\infty)\to\R$ is a smooth compactly supported function with\footnote{We denote the derivatives of $f$ with $\dot f, \ddot f$.} $\dot f(0)=0$.
    Let $0\in U\subseteq\R$ be an open set and let $a:U\times U\to\R$ be a $C^1$-function such that $a(0, 0) = 1$. The function $a$ is considered fixed, hence any constant that depends on $a$ can be hidden in $\omega(\eps)$ and in the big-$\bigo$ notation.
    
    If $\norm{u}_{C^1}\le \eps$, then it holds
    \begin{equation*}
        \int_{\R^{n-1}} a(u, \abs{\nabla u}^2) \nabla^2 u[\nabla u,\nabla u]
        =
        -\frac{(n-2)\Haus^{n-2}(\S^{n-2})}3\int_0^{\infty} \big(1+\omega(\eps)\big) \dot f^3 r^{n-3} \de r + \bigo\Big(\int_0^{\infty} \dot f^2 r^{n-2}\de r\Big) \,.
    \end{equation*}
\end{lemma}
\begin{proof}
    It holds $\abs{\nabla u} = \abs{\dot f}$ and $\nabla^2 u[\nabla u, \nabla u]=\ddot f\dot f^2$, thus we have
    \begin{equation}\label{eq:loc4431}
        \int_{\R^{n-1}} a(u, \abs{\nabla u}^2) \nabla^2 u[\nabla u,\nabla u]
        =
        \Haus^{n-2}(\S^{n-2})\int_0^\infty a(f, \dot f^2) \ddot f\dot f^2 r^{n-2} \de r
        \,.
    \end{equation}
    Let $b:U\times U\to\R$ be the function such that $b(s, 0)=0$ and $\partial_t b(s, t) = t^2a(s, t^2)$. From its definition, it follows that $b$ satisfies
    \begin{equation}\label{eq:loc99482}
        b(s, t) = \frac{t^3}3\big(1+\bigo(s) + \bigo(t^2)\big) \,.
    \end{equation}
    Thanks to the identity
    \begin{equation*}
        \frac{\de}{\de r}(b(f, \dot f)) = a(f, \dot f^2)\dot f^2\ddot f + \bigo(\abs{\dot f}^2)\,,
    \end{equation*}
    the equality \cref{eq:loc4431} becomes
    \begin{equation*}
        \int_{\R^{n-1}} a(u, \abs{\nabla u}^2) \nabla^2 u[\nabla u,\nabla u]
        =
        \Haus^{n-2}(\S^{n-2})\int_0^\infty \frac{\de}{\de r}(b(f, \dot f)) r^{n-2} \de r
        +\bigo\Big(\int_0^{\infty} \dot f^2 r^{n-2}\de r\Big)
        \,,
    \end{equation*}
    which, recalling \cref{eq:loc99482}, implies the statement after an integration by parts.
\end{proof}

\begin{proposition}\label{prop:curv_is_necessary}
    Given $n\ge 3$, for any $\eps>0$, there is domain $K\subseteq\R^n$ which is a $C^1$ $\eps$-perturbation of a ball and such that
    \begin{equation*}
        \int_{\partial K} H < -1 \,.
    \end{equation*}
\end{proposition}
\begin{proof}
    Here we say that a function on the sphere is radial if it depends only on the distance from a point $\bar x\in\S^{n-1}$, in such case we say that $\bar x$ is the origin of the radial function.
    
    Given a large integer $\kappa \gg \eps^{-3}$, take $q=c\kappa^{n-1}$ points $(x_i)_{i=1,\dots, q}$ on $\S^{n-1}$ (with $c=c(n)$ universal) such that for any $i\not= j$ we have $\abs{x_i-x_j}\ge 2\kappa^{-1}$. 
    
    Consider a function $f:[0, \infty)\to\R$ supported in $[0,\kappa^{-1}]$ such that $-\eps/2\le f\le 0$, $0\le \dot f\le \frac\eps2$ and $\dot f=\frac\eps2$ in the interval $[\frac18\kappa^{-1},\frac78\kappa^{-1}]$. 
    
    For any $1\le i\le q$, let $u_i:\S^{n-1}\to\R$ be the radial function with profile given by $f$ and origin in $x_i$. The functions $(u_i)_{i=1,\dots,q}$ have disjoint supports and their supports are small balls of radius $\kappa^{-1}$. 
    Furthermore, notice that $\norm{u_i}_{C^1}\le \eps$.
    Applying\footnote{To be precise, the lemma cannot be applied verbatim as here we are considering an integral on the sphere. Since the support of the function is very small, the curvature plays a minor role and the computations of the lemma hold up to a minor multiplicative correction that goes to $0$ as $\kappa$ goes to $\infty$.} \cref{lem:radial_int_parts} with $a(s,t) = \frac{(1+s)^{n-3}}{(1+s)^2+t}$, we obtain that for any $1\le i\le q$
    \begin{equation}\label{eq:loc92021}
        \int_{\S^{n-1}}(1+u_i)^{n-3}\frac{\nabla^2 u_i[\nabla u_i,\nabla u_i]}{(1+u_i)^2 + \abs{\nabla u_i}^2} = -\alpha \eps^3\kappa^{-(n-2)} + \beta \eps^2 \kappa^{-(n-1)} \,,
    \end{equation}
    where $\alpha > \alpha_0(n) > 0$ and $\abs{\beta} < \beta_0(n)$ are appropriate constants that depend on the precise choice of the function $f$.
    
    Let $u=\sum_i u_i$ be the sum of all the functions $u_i$. Let $K\subseteq\R^n$ be the star-shaped domain such that $(1+u(x))x\in\partial K$ for all $x\in\S^{n-1}$. The domain $K$ is a $C^1$ $\eps$-perturbation of the ball.

    Thanks to \cref{eq:integral_H_formula} and \cref{eq:loc92021}, we have
    \begin{align*}
        \int_{\partial K} H 
        &= \bigo(1) + \int\limits_{\S^{n-1}} (1+u)^{n-3}\frac{\nabla^2 u[\nabla u,\nabla u]}{(1+u)^2 + \abs{\nabla u}^2} = \bigo(1) + \sum_{i=1}^{c\kappa^{n-1}}\int\limits_{\S^{n-1}}(1+u_i)^{n-3}\frac{\nabla^2 u_i[\nabla u_i,\nabla u_i]}{(1+u_i)^2 + \abs{\nabla u_i}^2} \\
        &=
        \bigo(1) - c\kappa^{n-1}\alpha \eps^3 \kappa^{-(n-2)} + c\kappa^{n-1}\beta\eps^2 \kappa^{-(n-1)}
        =
        \bigo(1) - c\alpha\eps^3\kappa \,.
    \end{align*}
    If $\kappa$ is chosen sufficiently large, the desired statement follows.
\end{proof}

\printbibliography[heading=bibintoc]

\end{document}